\documentclass{article}
\usepackage[utf8]{inputenc}
\usepackage{amsfonts}
\usepackage{amsmath}
\usepackage{amssymb}
\usepackage{amsthm}
\usepackage{xcolor}
\usepackage{bm}
\usepackage{mathtools}
\usepackage{enumitem}
\usepackage{verbatim}
\usepackage{pinlabel}
\usepackage{caption}
\usepackage{subcaption}
\usepackage{tikz}
\usepackage{graphicx}
\usepackage{hyperref}
\usepackage{tikz-cd}
\usepackage{float}
\usepackage{geometry}
\newtheorem{theorem}{Theorem}[section]

\newtheorem{lemma}[theorem]{Lemma}
\newtheorem{proposition}[theorem]{Proposition}

\newtheoremstyle{claim}
  {\topsep}
  {\topsep}
  {}
  {}
  {\itshape}
  {.}
  {.5em}
  {\thmname{#1}\thmnumber{ #2}\thmnote{ (#3)}}
\theoremstyle{claim}

\DeclareMathOperator{\Homeo}{Homeo}
\DeclareMathOperator{\Diffeo}{Diff}
\DeclareMathOperator{\MCG}{MCG}

\DeclareMathOperator{\Aut}{Aut}
\DeclareMathOperator{\link}{link}
\DeclareMathOperator{\ann}{Ann}
\DeclareMathOperator{\Q}{\mathcal{Q}^{sep}}
\DeclareMathOperator{\sep}{sep}
\DeclareMathOperator{\im}{im}

\newcommand{\fine}{\mathcal{C}^\dagger}
\newcommand{\onefine}{\mathcal{C}^\dagger_1}
\newcommand{\onetrans}{\mathcal{C}^\dagger_\pitchfork}
\newcommand{\homotop}{\simeq}
\newcommand{\p}[1]{\medskip\noindent\textbf{#1}\textbf{.}}
\newcommand{\calH}{\mathcal{H}}
\newcommand{\pit}[1]{\medskip\noindent\textit{#1}\textit{.}}

\title{Automorphisms of the fine 1-curve graph}
\author{Katherine Williams Booth, Daniel Minahan, and Roberta Shapiro}
\date{}

\begin{document}

\vspace{-2in}
\maketitle

\begin{abstract}
    The fine 1-curve graph of a surface is a graph whose vertices are simple closed curves on the surface and whose edges connect vertices that intersect in at most one point. We show that the automorphism group of the fine 1-curve graph is naturally isomorphic to the homeomorphism group of a closed, orientable surface with genus at least one.
\end{abstract}

\section{Introduction}

Let $S = S_{g}$ be an oriented, connected, closed surface with genus $g$. The \textit{fine 1-curve graph of $S$}, denoted $\onefine(S)$, is a graph whose vertices are simple, closed, essential curves in $S$.  There is an edge between two vertices $u$ and $v$ if $|u \cap v|\leq 1.$  Since homeomorphisms preserve intersections, there is a natural homomorphism $\Homeo(S) \rightarrow \Aut\onefine(S)$ induced by the standard action of $\Homeo(S)$ on $S$.
Our main result is the following.

\begin{theorem}\label{maintheorem}
    Let $S_g$ be a closed, orientable, connected surface with $g \geq 1$. The map 
    \begin{displaymath}
    \Phi:\Homeo(S_g) \rightarrow \Aut\onefine(S_g)
    \end{displaymath}
    \noindent is an isomorphism.
\end{theorem}

A version of the fine 1-curve graph, denoted $\mathcal{C}^\dagger_\pitchfork(S),$ was introduced by Le~Roux--Wolff \cite{transaut}. In their paper, they work with connected, nonspherical and possibly nonorientable or noncompact surfaces. The vertices of their graph correspond to nonseparating curves while edges connect pairs of curves that are either disjoint or intersect once essentially (termed torus pairs below). Le~Roux--Wolff show that $\Aut \mathcal{C}^\dagger_\pitchfork(S)$ is isomorphic to $\Homeo(S)$. 

\p{The fine curve graph} The \textit{fine curve graph} of $S,$ denoted $\fine(S)$, was introduced by Bowden--Hensel--Webb to study $\Diffeo_0(S)$ \cite{BHW}. The vertices of $\fine(S)$ are essential, non-peripheral, simple, closed curves in $S$.  Two vertices are connected by an edge if their corresponding curves are disjoint. In analogy with our theorem, Long--Margalit--Pham--Verberne--Yao prove that $\Aut \fine(S)$ is isomorphic to $\Homeo(S)$ \cite{fineaut}.

\p{Graphs of curves and Ivanov's metaconjecture} A classically studied object is the \textit{curve graph} of $S$, denoted $\mathcal{C}(S)$.  The vertices of $\mathcal{C}(S)$ are isotopy classes of essential, non-peripheral (if $S$ has boundary), simple closed curves in $S$.  Two vertices are connected by an edge if they admit disjoint representatives. The \textit{extended mapping class group} of $S,$ denoted $\MCG^{\pm}(S),$ is the group of connected components of $\pi_0
(\Homeo(S))$. Ivanov showed that, for surfaces of genus at least three, $\Aut\mathcal{C}(S)$ is isomorphic to $\MCG^\pm(S)$ \cite{ivanov}.  Following this, Ivanov made the following metaconjecture \cite[pg 84]{Farbproblems}.

\pit{Ivanov's metaconjecture} Every object naturally associated to a surface S and having a sufficiently rich structure has $\MCG^{\pm}(S)$ as its groups of automorphisms. Moreover, this can be proved by a reduction to the theorem about the automorphisms of $\mathcal{C}(S).$
\color{black}


\medskip

\noindent Brendle and Margalit showed that Ivanov's metaconjecture holds for a large number of graphs where edges correspond to disjointness \cite[Theorem 1.7]{BrendleMargalit}. 

\p{The $k$-curve graph} Ivanov's metaconjecture may also hold for graphs of curves where edges do not correspond to disjointness. For example, consider the \textit{$k$-curve graph}, $\mathcal{C}_k(S_g),$ which has the same vertices as the curve graph.  Edges connect vertices whose isotopy classes admit representatives that intersect at most $k$ times.  Agrawal--Aougab--Chandran--Loving--Oakley--Shapiro--Xiao \cite{kcurve} showed that when $g$ is sufficiently large with respect to $k$, $\Aut \mathcal{C}_k(S_g)$ is isomorphic to $\MCG^\pm(S_g)$ for any $k \geq 1$. 

Similarly to how the fine curve graph is an analogue of the curve graph, the fine 1-curve graph is an analogue of the $k$-curve graph for $k=1$.

\p{Sketch of the proof of Theorem \ref{maintheorem}} When $g \geq 2$, we prove Theorem~\ref{maintheorem} by reducing to the theorem of Long--Margalit--Pham--Verberne--Yao \cite{fineaut}. In particular, we show that every automorphism of $\onefine(S)$ preserves the set of edges connecting disjoint curves. For $g=1$,  we reduce to the theorem of Le~Roux--Wolff \cite{transaut} and show that the set of edges that are in $\onefine(S)$ but not in $\onetrans(S)$ are preserved by automorphisms. Edges in this set correspond to pairs of curves in a specific configuration; such a pair of curves will be called a pants pair and is defined in the following paragraph. 

\pit{Torus pairs versus pants pairs} There are two types of configurations of pairs of curves that intersect once. If a pair of curves crosses at their point of intersection, we call it a \textit{torus pair}, as on the left side of Figure \ref{torusvspantsfig}. We note that both curves that comprise a torus pair must be nonseparating. Otherwise, if neither curve crosses the other at their point of intersection, we call it a \textit{pants pair}, as on the right hand side of Figure \ref{torusvspantsfig}. These definitions are reminiscent of those in Long--Margalit--Pham--Verberne--Yao \cite{fineaut}, with two key differences: we require all intersections to be single points (called \textit{degenerate} in \cite{fineaut}) and the curves in a pants pair are allowed to be homotopic. 

\begin{figure}[h]
\begin{center}
\includegraphics[width=1.5in]{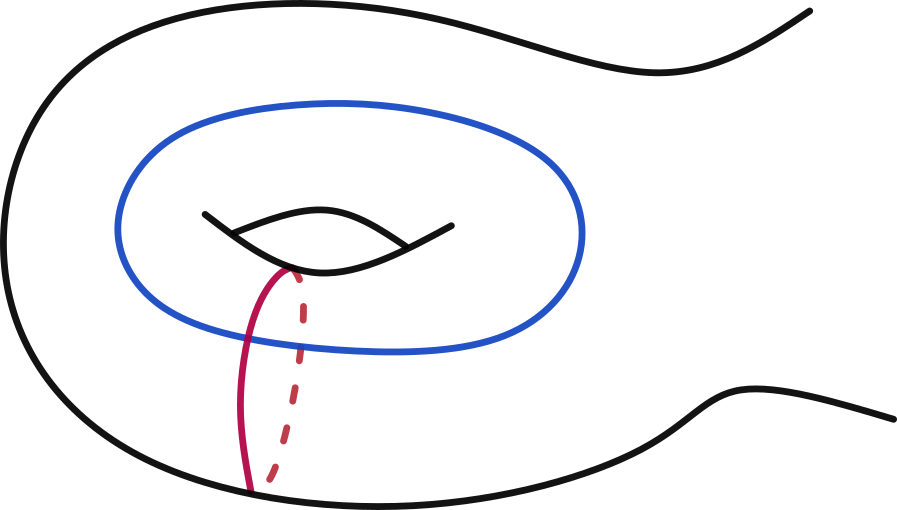}
\hspace{.2in}
\includegraphics[width=1.5in]{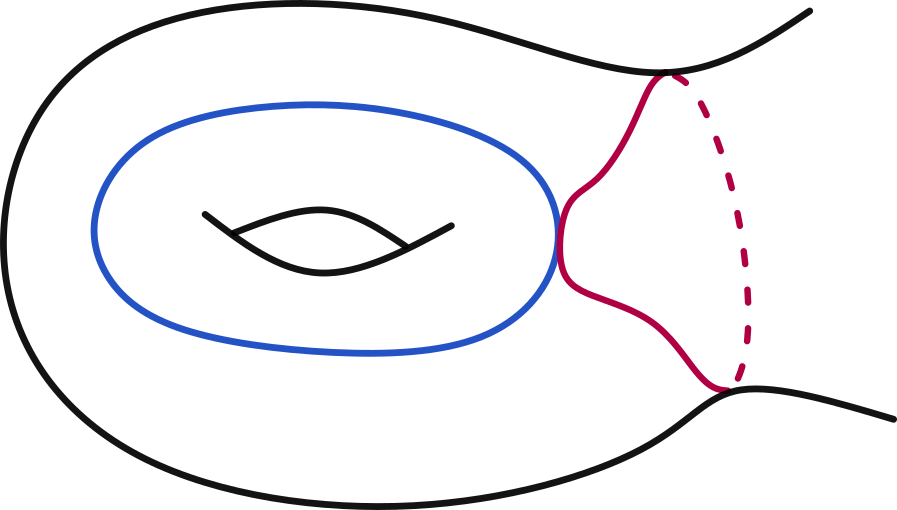}
\caption{Examples of torus pairs (left) and pants pairs (right)}\label{torusvspantsfig}
\end{center}
\end{figure}

\p{Paper outline} In Section~\ref{sec:separatingcurves}, we prove several preliminary results about separating curves. In Section \ref{toruspairssection}, we show that torus pairs are preserved by automorphisms of $\onefine(S)$. In Section~\ref{pantspairssection}, we show that pants pairs are preserved by automorphisms when $g \geq 2$.  In Section \ref{torussection}, we prove Theorem \ref{maintheorem} in the case that $g = 1$.  We conclude by proving Theorem \ref{maintheorem} in Section~\ref{mainsection}.  

\p{Acknowledgments} The authors would like to thank their advisor Dan Margalit for many helpful conversations. The authors would also like to thank Jaden Ai, Ryan Dickmann, Jacob Guynee, Sierra Knavel, and Abdoul Karim Sane for useful discussions. The authors thank F\'ed\'eric Le~Roux and Maxime Wolff for sharing their manuscript and further correspondences. The authors further thank Nick Salter for comments on a draft of the manuscript. The first author was supported by the National Science Foundation under Grant No. DMS-1745583. The third author was partially supported by the National Science Foundation under Grant No. DMS-2203431.

\section{Separating curves and their homotopy classes}\label{sec:separatingcurves}

In this section, we give several prelimininary results about separating curves. We prove in Lemma \ref{separating} that separating curves are preserved by automorphisms of $\onefine(S_g)$. In Lemma~\ref{sepdisjointhom}, we prove that pairs of homotopic separating curves that are adjacent in $\onefine(S_g)$ are preserved by automorphisms of $\onefine(S_g)$. 

We also define a new quotient graph that will be used in future sections. The relation used to define the quotient graph is proven to be equivalent to homotopy of curves in Lemma~\ref{equividhomo}. Moreover, we show that the structure of this quotient graph is preserved by automorphisms of $\onefine(S_g)$ in Lemma~\ref{seplinkaction}.

\p{Preliminary graph theoretic definitions} Two vertices connected by an edge in a graph are called \textit{adjacent}. The \textit{link} of a vertex $v$ in a graph $G$, denoted $\link(v)$, is the subgraph induced by the vertices adjacent to $v$ in $G$. A graph $G$ is a \textit{join} if there is a partition of the vertices of $G$ into at least two nonempty subsets, called \textit{parts}, such that every vertex in one part is adjacent to all vertices in all of the other parts. A graph $G$ is an $n$-\textit{join} if it is a join that can be partitioned into $n$ sets but cannot be partitioned into $n+1$ sets. A graph $G$ is a \textit{cone} if there is a vertex $v,$ called a \textit{cone point}, that is adjacent to all other vertices in $G.$

\medskip\noindent We say that a separating curve $u$ \textit{separates} the curves $a$ and $b$ if $a$ and $b$ lie in the closures of distinct connected components of $S\setminus u.$ If $a$ and $b$ are contained in the closure of the same connected component of $S\setminus u$, then they are on the same side of $u;$ otherwise, $a$ and $b$ are on different sides of $u.$

We begin by showing that the sets of separating and nonseparating curves are each preserved by automorphisms of $\onefine(S_g)$.

\begin{lemma}\label{separating}
    Let $\varphi\in \Aut\onefine(S_g)$. Then $u$ is a separating curve if and only if $\varphi(u)$ is a separating curve. Moreover, $\varphi$ preserves the sides of $u$.
\end{lemma}

\begin{proof}
    We will show that a curve $u$ is separating if and only if $\link(u)$ is a join. In fact, we will show that it is a 2-join.

    Suppose $u$ is separating. Then no curve in $\link(u)$ can form a torus pair with $u$ and must be either disjoint or form a pants pair with $u$. Therefore, every curve in $\link(u)$ will lie in the closure of exactly one of the components of $S_g\setminus u$. Let $A$ and $B$ be the closures of the two components of $S_g\setminus u.$ 
    
    We claim that $\link(u)$ is a join with two parts are given by curves contained in $A$ and the curves contained in $B.$ Suppose $a$ and $b$ are curves in $\link(u)$ such that $a\subset A$ and $b\subset B.$ Since $|a\cap u|\leq 1$ and $|b\cap u|\leq 1,$ and $(a\cap b) \subset u,$ we have that $|a\cap b|\leq 1.$ We conclude that $a$ and $b$ are adjacent in $\onefine(S_g),$ thus concluding the proof of the claim. 

    Suppose $u$ is not separating. Then, $S_g\setminus u$ is a single connected component. Let $a,b\in \link(u).$ Then, we can move $a$ off of itself and isotope it to intersect $a$ and $b$ at least twice each and $u$ at most once; call this new curve $a'.$ It follows that $a'$ is adjacent to neither $a$ nor $b$, so $a$ and $b$ cannot be in different parts of a join. We conclude that $\link(u)$ cannot be a join.
    
    If two curves $a,b\in\link(u)$ lie in the closure of the same component of $S_g\setminus u$, they must lie in the same part of the join $\link(u)$. It follows that $\link(u)$ is a 2-join. Thus, $\varphi(a)$ and $\varphi(b)$ are in the same part of the join $\link(\varphi(u))$ and therefore lie in the closure of the same component of $S\setminus\varphi(u).$
\end{proof}


\p{Link of an edge and the separating link} Let $u$ and $v$ be adjacent vertices in $\onefine(S).$ The \textit{link of an edge} spanned by $u,v$ is $\link(u,v)=\link(u)\cap \link(v).$ The \textit{separating link of} $(u,v)$, denoted $\link^{\text{sep}}(u,v),$ is the subgraph of $\link(u,v)$ induced by separating curves.


We use these definitions to show that homotopic pairs of separating curves adjacent in $\onefine(S_g)$ are preserved by automorphisms of $\onefine(S_g).$

\begin{lemma}\label{sepdisjointhom}
    Let $u$ and $v$ be adjacent separating curves in $\onefine(S_g)$. Then $u$ and $v$ are homotopic if and only if $\link(u,v)$ is a 3-join such that one of the parts contains only separating curves.  Hence, for any $\varphi \in \Aut\onefine(S)$, a separating curve $u \in \onefine(S)$ is homotopic to an adjacent curve $v \in \onefine(S)$ if and only if $\varphi(u)$ is homotopic to $\varphi(v)$.
\end{lemma}

\begin{proof}
    Because $u$ and $v$ are both separating and are either disjoint or form a pants pair, $S\setminus(u\cup v)$ must have three connected components. By a similar argument to Lemma~\ref{separating},  $\link(u,v)$ is a 3-join induced by said components.  Moreover, the link of a pair of adjacent nonseparating curves or a nonseparating and separating adjacent pair of curves is not a 3-join. This follows directly using the same argument that the link of a nonseparating curve is not a join and from the fact that such pairs will not split a surface into three connected components. 

    If $u$ and $v$ are homotopic as in Figure \ref{homosepfig}, then one of the components of $S\setminus(u\cup v)$ is a (possibly pinched) annulus. All essential curves that lie in the (possibly pinched) annulus are separating, so it follows that one of the parts is comprised of separating curves.
    
    \begin{figure}[h]
\begin{center}
\begin{tikzpicture}
    \node[anchor=south west, inner sep = 0] at (0,0){\includegraphics[width=3in]{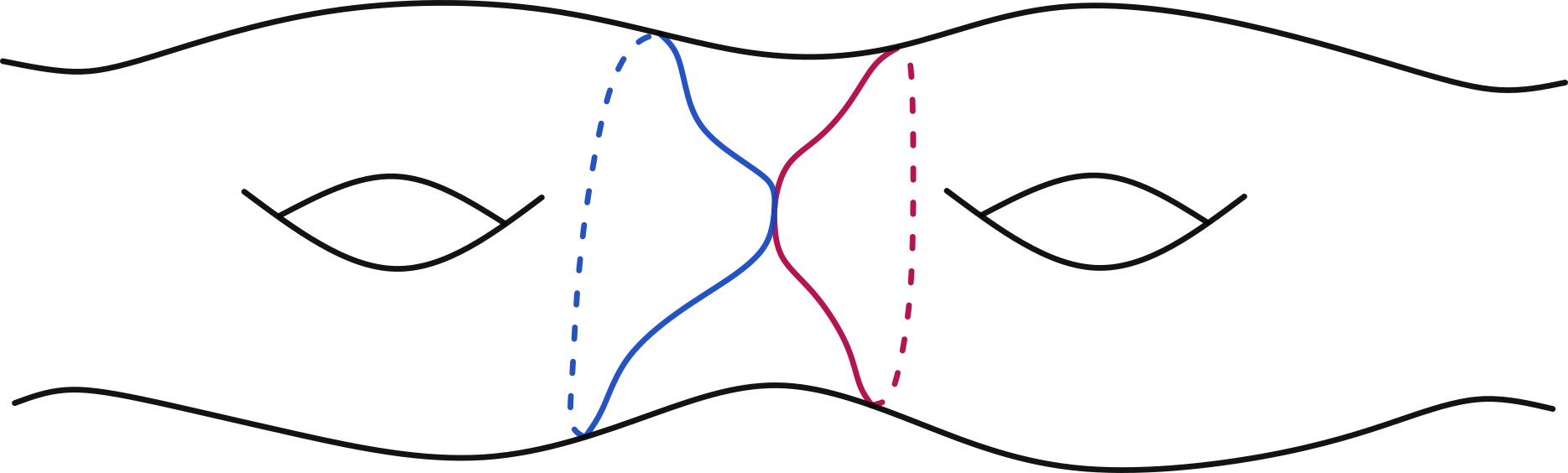}};
    \node at (3,1) {$u$};
    \node at (4.6,1) {$v$};
\end{tikzpicture}
    \caption{Adjacent, homotopic separating curves}\label{homosepfig}
\end{center}
\end{figure}
    If $u$ and $v$ are not homotopic as in Figure \ref{nonhomosepfig}, then $S\setminus(u\cup v)$ consists of three components, each of which has genus. Therefore, each component supports a nonseparating curve. It follows that all parts of the join contain a nonseparating curve. 
    \begin{figure}[h]
\begin{center}
\begin{tikzpicture}
    \node[anchor = south west, inner sep = 0] at (0,0){\includegraphics[width=3in]{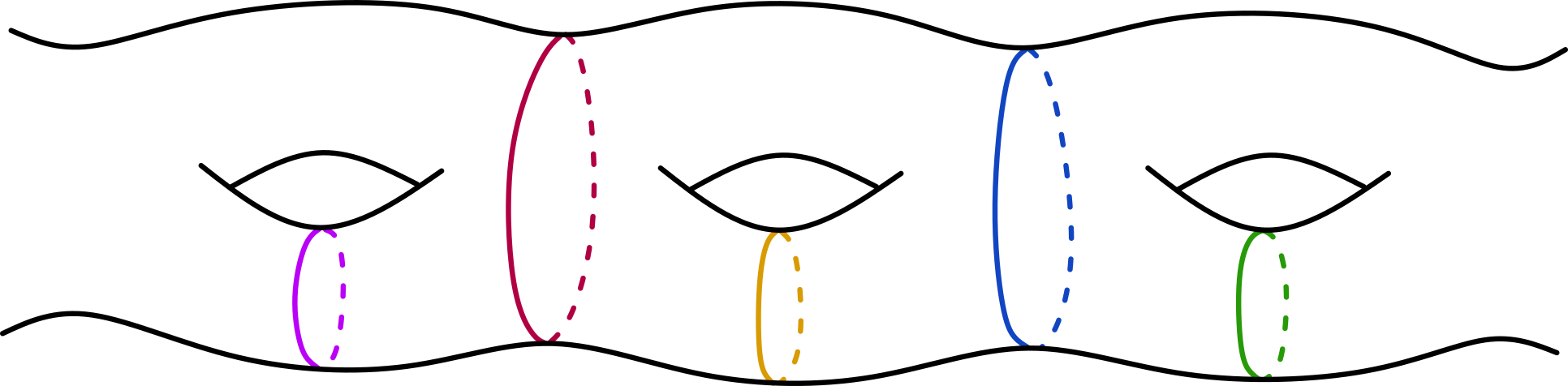}};
    \node at (2.3,0.7) {$u$};
    \node at (4.7,0.7) {$v$};
    \end{tikzpicture}
    \caption{The link of adjacent, non-homotopic separating curves is a 3-join where each part contains a nonseparating curve.}\label{nonhomosepfig}
\end{center}
\end{figure}
\end{proof}

With this in mind, we prove the following result about homotopic curves in the link of an edge. The original idea behind the proof is in Bowden--Hensel--Webb \cite{BHW} and is expanded on in Long--Margalit--Pham--Verberne--Yao \cite{fineaut}.

\begin{lemma}\label{equividhomo}
Let $u,v$ be two nonseparating adjacent curves in $\onefine(S)$.  Let $a$ and $b$ be two separating curves in $\link(u,v)$.  Then $a$ and $b$ are homotopic if and only if there is a path from $a$ to $b$ in $\link(u,v)$ consisting of curves homotopic to $a$ and $b$.
\end{lemma}

We need one auxiliary result before proving Lemma~\ref{equividhomo}.

\begin{lemma}\label{disjisolemma}
Let $u,v$ be two nonseparating adjacent curves in $\onefine(S)$.  Let $a$ and $b$ be two separating, homotopic curves in $\link(u,v)$.  Then there is a homotopy $\psi:S^1 \times I \rightarrow S_g$ from $a$ to $b$ such that, for every $t \in I$, there is a closed neighborhood $T \subseteq I$ with $t \in T$ where $\psi(S^1 \times T)$ is contained in a possibly pinched annulus $A_T$ such that $u$ intersects $A_T$ at most once, and $v$ intersects $A_T$ at most once.
\end{lemma}

\begin{proof}
There are three cases to consider.

\pit{Case 1: $u$ and $v$ are on the same side of $a$} Since $a$ and $b$ are isotopic, $u$ and $v$ must be on the same side of $b$ as well.  Choose an annulus $A$ with $a \subseteq \partial A$ such that $A$ is supported on the side of $a$ that does not contain $u$ and $v$.  Then $a$ is homotopic to the other boundary component $a'$ of $A$, and each curve in this homotopy is disjoint from $u$ and $v$ (besides $a$ itself).  Similarly, construct $b'$ adjacent to $b$.  Then $a'$ and $b'$ are disjoint from $u$ and $v$, so there is a homotopy between $a'$ and $b'$ that consists only of vertices of $\link(u,v)$ disjoint from $u$ and $v$.  Let $\psi:S^1 \times I$ be the resulting homotopy from $a$ to $b$.  Then for any $t \in I$, there is an annulus $A_t$ containing $\psi(S^1 \times \{t\})$ that contains $\psi(S^1 \times T)$ for some closed neighborhood of $t$, and that only intersects $u$ and $v$ each in at most one point, namely either the points of intersection of $a$ with $u$ and $v$, or the points of intersection of $b$ with $u$ and $v$, so the lemma holds.

\pit{Case 2: $u$ and $v$ are on opposite sides of $a$, and $u$ and $v$ are disjoint} As in Case 1, it suffices to show that $a$ is homotopic to $a' \in \link(u,v)$ disjoint from $u$ and $v$ via vertices of $\link(u,v)$.  Since $u$ and $v$ are disjoint, we can choose closed annuli $A_u$ and $A_v$ such that $u \subseteq A_u$, $v \subseteq A_v$, and $A_u \cap a$, $A_v \cap a$ are each either connected or empty.  If $A_u \cap a$ is nonempty, let $d_u$ be the unique subinterval of $\partial A_u$ that connects the two points of intersection $a \cap A_u$ to each other, such that the path $A_u \cap a$ is homotopic rel $\partial A_u \cap a$ to $d_u$.  Let $a''$ be the curve given by removing $a \cap A_u$ from $a$ and replacing it with $d_u$.  Construct $d_v$ similarly, and let $a'$ be the resulting curve given by removing $a'' \cap A = a \cap A$ and replacing it with $d_v$.  Then $a$ is homotopic to $a'$ and $a'$ is disjoint from $u$ and $v$.  Repeat this same process for $b$ to get $b'$, and then $a'$ and $b'$  are homotopic through vertices disjoint from $u$ and $v$.  Similarly to Case I, this completes the proof of the Lemma in this case.

\pit{Case 3: $u$ and $v$ are on opposite sides of $a$, and $u$ and $v$ intersect} Let $\delta$ be the unique point of intersection in $u \cap v$.  Since $u$ and $v$ are on opposite sides of $a$, we see that $\delta \in a$.  Now, since $b$ is homotopic to $a$, it must be the case that $u$ and $v$ are on opposite sides of $b$ as well.  Therefore, $\delta \in b$ as well.  Now, we can think of $\delta$ as a marked point and identify $a$, $b$, $u$, and $v$ with arcs $\overline{a},\ \overline{b},\ \overline{u},$ and $\overline{v},$ respectively, based at $\delta$.  We see that $\overline{a}$ and $\overline{b}$ are disjoint from $\overline{u}$ and $\overline{v}$, since as curves $a$ and $b$ only intersect $u$ and $v$ each at most once.  Hence the arcs $\overline{a}$ and $\overline{b}$ are homotopic to each other in the marked surface $(S_g, \delta)$ via a path of arcs disjoint from $\overline{u}$ and $\overline{v}$.  But then this homotopy of arcs is canonically identifies with a homotopy of curves in $\link(u,v)$ that all intersect $u$ and $v$ only at $\delta$.  Let $\psi$ be the resulting homotopy.  For any $t \in I$, there is a pinched annulus $A$ containing $\psi(S^1 \times T)$ for some closed neighborhood $t \in T$, such that $A$ intersects $u$ and $v$ only at $\delta$, so the lemma holds.
 \end{proof}

\begin{proof}[Proof of Lemma \ref{equividhomo}]
The backwards direction follows by definition, so we only need to prove the forwards direction.  Let $\psi:S^1 \times I \rightarrow S_g$ be a homotopy as in Lemma \ref{disjisolemma}.  For each point $t_i \in I$, choose an interval $T_i = [s_i, s'_i] \subseteq I$ with $s_i \neq s'_i$ that contains $t_i$ such that the set $\psi(S^1 \times T_i)$ is contained in a (possibly pinched) annulus $A_i$, such that $A_i$ intersects $u$ and $v$ in at most one point each.  Since $I$ is compact, there is a finite  collection of such intervals $T_1,\ldots, T_n$ whose interiors cover $I$. 
 By restricting each $T_i$, we may assume that $s'_i = s_{i+1}$ for all $0 \leq i < n$.  

Observe that if two separating curves $x,y \subseteq S_g$ are contained in the interior of a (possibly pinched) annulus $A$, then a curve $z\subset \partial A$ is both adjacent and homotopic to $x$ and $y$, forming a path of length two between $x$ and $y.$  For each $T_i$, there are two curves $x_i=\psi(S^1\times \{s_i\})$ and $y_i=\psi(S^1\times \{s'_i\})$ given by the restriction of $\psi$ to each endpoint of $T_i$ such that $x_i$ and $y_i$ both each intersect $u$ and $v$ in at most one point, since we have assumed the same for $A_i$.  Hence each $x_i$ and $y_i$ are vertices in $\link(u,v)$.  Then each $x_i$ is connected by a path of length 2 in $\link(u,v)$ to $y_i$.  We have chosen the $T_i$ in such way that $y_i = x_{i+1}$ for $0 \leq i < n$.  Then $x_0 = a$ and $y_n = b$ by construction, so there is a path from $x_0$ to $y_n$ of length $2n$ in $\link(u,v)$, such that each vertex of this path is homotopic to $a$, so the lemma holds.
\end{proof}

\p{Separating link quotient} Let $u$ and $v$ be adjacent nonseparating curves  in $\onefine(S_g).$ Define the \textit{separating link quotient of $(u,v)$}, denoted $\Q (u,v),$ to be the separating link of $u$ and $v$ quotiented out by homotopy on the vertices. In other words, the vertices of $\Q(u,v)$ are homotopy classes of separating curves in $\link(u,v)$ and an edge connects two vertices if they admit disjoint representatives.  Hereafter, if $a$ is a separating curve in $\onefine(S_g)$, we will denote two things by $[a]$: 1) the set of curves in $\onefine(S_g)$ homotopic to $a$ and 2) the vertex that $a$ represents in $\Q(u,v)$.

In the following lemma, we prove that the structure of the separating link quotient is preserved under automorphisms of $\onefine(S_g).$

\begin{lemma}\label{seplinkaction}
    Let $u$ and $v$ be nonseparating curves adjacent in $\onefine(S_g)$ and let $\varphi \in \Aut\onefine(S_g)$.  Then $\Q(u,v) \cong \Q(\varphi(u), \varphi(v))$.  
\end{lemma}
\begin{proof}
    By Lemma \ref{separating}, $\varphi$ induces an isomorphism between $\link^{\sep}(u,v)$ and $ \link^{\sep}(\varphi(u), \varphi(v))$.  Then by Lemma \ref{sepdisjointhom}, $\varphi$ preserves disjoint homotopic separating curves. Hence if $d, d' \in \link^{\sep}(u,v)$ map to the same point in $\Q(u,v)$, then $\varphi(d)$ and $ \varphi(d')$ must map to the same point in $\Q(\varphi(u), \varphi(v))$.  Therefore the isomorphism $\link^{\sep}(u,v) \cong \link^{\sep}(\varphi(u), \varphi(v))$ descends to an isomorphism $\Q(u,v) \cong \Q(\varphi(u), \varphi(v))$.
\end{proof}

With these preliminary results in mind, we are ready to proceed with the main body of the proof of Theorem~\ref{maintheorem}.

\section{Torus pairs}\label{toruspairssection}

In this section, we prove that automorphisms of our graph preserve torus pairs. 

\begin{proposition}\label{toruspairaut}
    Let $S_g$ be a surface of genus $g \geq 2$.  Let $\varphi \in \Aut \onefine(S_g)$ and $u, v$ be adjacent curves in $\onefine(S_g)$.  Then $\left(\varphi(u), \varphi(v)\right)$ is a torus pair if and only if $(u,v)$ is a torus pair.
\end{proposition}

We will now characterize torus pairs using the graph $Q^{\sep}(u,v)$ in Lemma~\ref{toruspairqsepcharlemma}.

\begin{lemma}\label{toruspairqsepcharlemma}
Let $g \geq 2$.  Let $u$ and $v$ be nonseparating, adjacent curves in $\onefine(S_g)$.  Then $(u,v)$ is a torus pair if and only if the separating link quotient $\Q(u,v)$ is a cone.
\end{lemma}
\begin{proof}
We first assume that $(u,v)$ is a torus pair and show that $Q^{\sep}(u,v)$ is a cone.  We then assume that $(u,v)$ is not a torus pair, and show that $Q^{\sep}(u,v)$ is not a cone.
    
\pit{Suppose $(u,v)$ is a torus pair}  Consider the homotopy class $\calH$ of curves in $\link(u,v)$ homotopic to the boundary of the torus that $u$ and $v$ fill. By Lemma \ref{equividhomo}, $\calH$ descends to a single point in $Q^{\sep}(u,v)$.

    \begin{figure}[h]
\begin{center}
\begin{tikzpicture}
    \node[anchor = south west, inner sep = 0] at (0,0) {\includegraphics[width=2in]{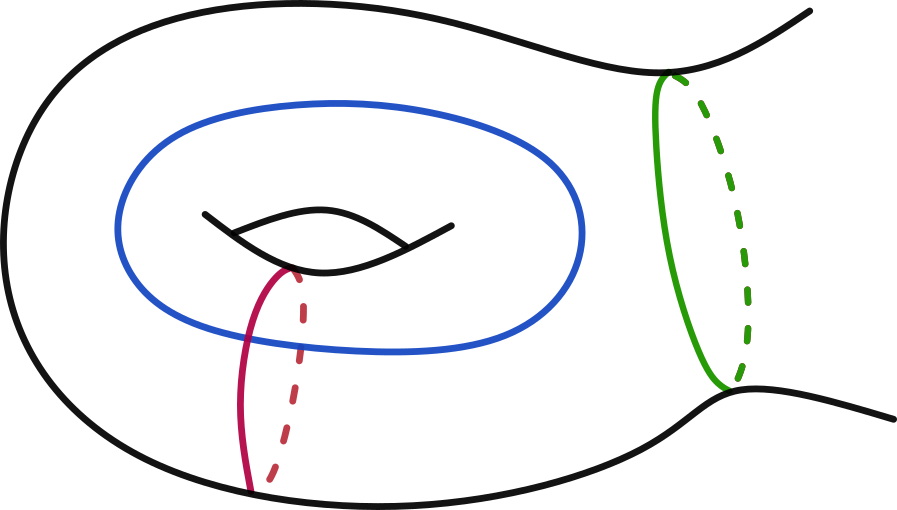}};
    \node at (1.8,0.5) {$u$};
    \node at (2,2.5) {$v$};
    \node at (4.3,2) {$\delta$};
    \end{tikzpicture}
    \caption{A torus pair with the separating curve $\delta$ which represents the cone point in $Q^{\sep}(u,v)$ }\label{toruscone}
\end{center}
\end{figure}
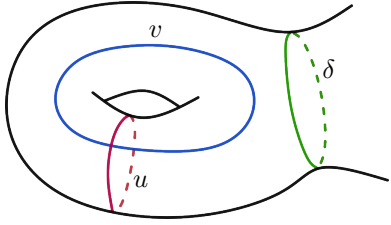  
    
    Let $a \in \link(u,v)\setminus \mathcal{H}$ be a separating curve. Then, $a$ is not contained in the torus filled by $u\cup v,$ and therefore is isotopic to a curve $a'$ disjoint from $u\cup v.$ Let $b$ be the boundary of a regular neighborhood of $u\cup v$ disjoint from $a'$. Thus $b\in\mathcal{H}$ and $b$ is disjoint from---and therefore adjacent to---$a'.$ We therefore have that $\calH$ is a cone point of $Q^{\sep}(u,v)$.

\pit{Suppose $(u,v)$ is not a torus pair} We must show that there is no cone point of $\Q (u,v)$. We will do this by ascertaining that for any homotopy class of separating curves in $\link(u,v),$ there is another homotopy class of curves that minimally intersects the first at least twice.

    Let $a\in\link(u,v)$ be a separating curve. Because $a$ is separating, there must be at least one genus in each connected component of $S_g\setminus a.$ Therefore, there is a nonseparating curve $c$ that: 1) intersects $a$ and cannot be homotoped to be disjoint from $a$, and 2) is disjoint from $u\cup v.$ 
    
    Consider a representative $b$ of $T_{[c]}[a],$ the Dehn twist of $[a]$ about $[c].$ The intersection number of $[a]$ and $[b]$ must be at least 2, so $|a\cap b|\geq 2$. Furthermore, since $c$ and $a$ are disjoint from $u$ and $v$, we can choose $b$ to be disjoint from $u$ and $v$ as well.  Hence, the vertices in $\Q (u,v)$ corresponding to $a$ and $b$ are not adjacent, and thus the vertex corresponding to $a$ is not a cone point.

    We conclude that $\Q (u,v)$ has no cone points.
 \end{proof}

 \begin{proof}[Proof of Proposition \ref{toruspairaut}]
Let $u,v$ be as in the statement of the proposition and let $\varphi \in \Aut\onefine(S_g)$.  Since $\varphi$ is invertible, it suffices to show that $(u,v)$ a torus pair implies $(\varphi(u), \varphi(v))$ a torus pair.  By Lemma \ref{toruspairqsepcharlemma}, $Q^{\sep}(u,v)$ is a cone.  Then by Lemma \ref{seplinkaction}, $Q^{\sep}(\varphi(u), \varphi(v))$ is a cone.  Therefore $(\varphi(u), \varphi(v))$ is a torus pair by Lemma \ref{toruspairqsepcharlemma}.
 \end{proof}

\section{Pants pairs}\label{pantspairssection}
The goal of this section is to prove the following proposition:

\begin{proposition}\label{allpants}
    Let $S_g$ be a surface of genus $g \geq 2$.  Let $\varphi \in \Aut\onefine(S_g)$ and $u, v$ be adjacent curves in $\onefine(S_g)$.  Then $\left(\varphi(u), \varphi(v)\right)$ is a pants pair if and only if $(u,v)$ is a pants pair.
\end{proposition}

The main observation behind the proof of Proposition~\ref{allpants} is that two curves $u$ and $v$ are disjoint if and only if $u$ has a neighborhood disjoint from $v$. To use this observation, we will break down the proof of Proposition~\ref{allpants} into several steps, as follows.
\begin{enumerate}[leftmargin=3\parindent]
    \item[Step 0.] Reduce to the case that $u$ and $v$ are both nonseparating (Lemma~\ref{seppantslemma}). 
    \item[Step 1.] Distinguish the boundary curves of a neighborhood of a nonseparating curve by showing that automorphisms preserve: 
    \begin{enumerate}
        \item [Step 1.1.] pairs of adjacent homotopic nonseparating curves  $\onefine(S_g)$ (Lemma~\ref{nonsephomotaut}),
        \item [Step 1.2.] whether a nonseparating curve is contained (in some sense) in the annulus bounded by adjacent homotopic curves $a$ and $b$ (Lemma \ref{lemma:annularmain}), and
        \item [Step 1.3.] whether two homotopic nonseparating curves adjacent in $\onefine(S_g)$ form a pants pair or are disjoint (Lemma~\ref{homnonseppants}).
    \end{enumerate}
    \item[Step 2.] Show that automorphisms preserve whether non-homotopic nonseparating curves are disjoint or a pants pair (Lemma~\ref{nonsepnonhompants}).
\end{enumerate}
Combining Steps 0, 1.3, and 2 proves Proposition~\ref{allpants} for all possible arrangements of curves.

We begin by proving Step 0 in the following lemma. We use the notation $A_1{*}A_2{*}\cdots{*}A_k$ to denote the decompositon of a $k$-join into its parts. 

\begin{lemma}\label{seppantslemma}
Let $u$ and $v$ be adjacent curves in $\onefine(S)$. Suppose that $u$ is separating. 

\begin{enumerate}[label=(\arabic*)]
    \item \textbf{If $\boldmath{v}$ is nonseparating:} Let $A {*} B$ be the decomposition of $\link(u)$ into a join and let $v\in A.$ Then $(u,v)$ is a pants pair if and only if there exists a nonseparating curve $w\in B$ such that $(v,w)$ is a pants pair.
    \item \textbf{If $\boldmath{v}$ is separating and homotopic to $\boldmath{u}$:} Let $A{*}B{*}C$ be a decomposition of $\link(u,v)$ into a 3-join such that $B$ contains only separating curves. Then, $(u,v)$ is a pants pair if and only if there exist nonseparating curves $a\in A$ and $c\in C$ such that $(a,c)$ is a pants pair. 
    \item \textbf{If $\boldmath{v}$ is separating and not homotopic to $\boldmath{u}$:} Let $A{*}B{*}C$ be a decomposition of $\link(u,v)$ into a 3-join. Then, $(u,v)$ is a pants pair if and only if there exist nonseparating curves $a\in A,\ b\in B,$ and $c\in C$ such that $(a,b),$ $(b,c),$ and $(c,a)$ are all pants pairs.
\end{enumerate}
\noindent In particular, if $\varphi \in \Aut \onefine(S)$ and $\varphi$ sends pants pairs consisting of nonseparating curves to pants pairs consisting of nonseparating curves, then $\varphi$ sends any pants pair to a pants pair.
\end{lemma}

\begin{proof}
    \pit{The case that $v$ is nonseparating} Suppose $(u,v)$ is a pants pair. Let $w\in B$ be a nonseparating curve that intersects $u$ at the point $u \cap v$. On the other hand, suppose $u$ and $v$ are disjoint. Then any curve $w \in B$ must be disjoint from $v$.
    
    \pit{The case that $v$ is separating and homotopic to  $u$} Suppose $(u,v)$ is a pants pair. Then, we can choose nonseparating curves $a\in A$ and $c\in C$ that intersect $u$ and $v$ at $u\cap v,$ as in Figure~\ref{seppantsmorefig}.
    
    Suppose $u$ and $v$ are disjoint. Then any curve in $A$ is disjoint from any curve in $C$. 
    
    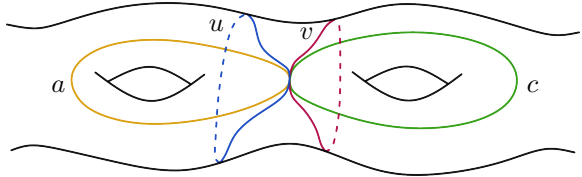
\begin{figure}[h]
        \centering
        \begin{tikzpicture}
        \node[anchor = south west, inner sep = 0] at (0,0){\includegraphics[width=3in]{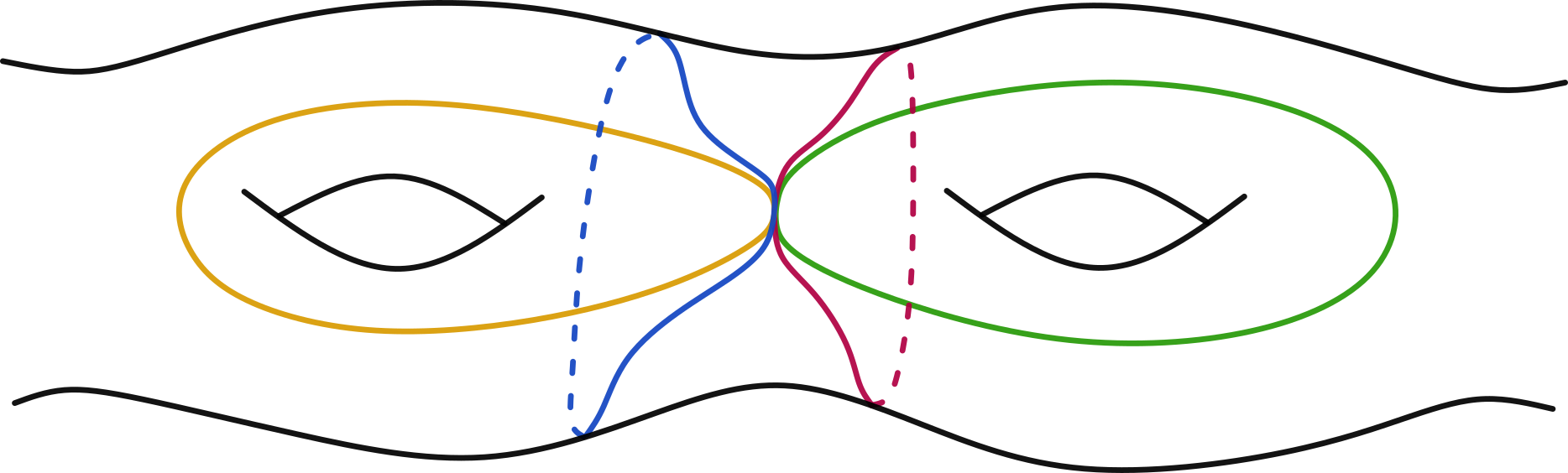}};
        \node at (2.8,2) {$u$};
        \node at (4,1.9) {$v$};
        \node at (0.7, 1.2) {$a$};
        \node at (7, 1.2) {$c$};
        \end{tikzpicture}
        \caption{If $u$ and $v$ form a pants pair, we can find nonseparating curves $a\in A$ and $c\in C$ that form a pants pair.}
        \label{seppantsmorefig}
    \end{figure}
    
    \pit{The case that $v$ is separating and not homotopic to $u$} Suppose $(u,v)$ is a pants pair. Then we can choose nonseparating curves $a\in A$, $b\in B,$ and $c\in C$ such that they pairwise form pants pairs. An example of such a selection of curves is in Figure~\ref{seppantsevenmorefig}.
    
    Suppose $u$ and $v$ are disjoint. Take $A$ to correspond to all curves supported in the subsurface bounded by only $u$ and $C$ to correspond to all curves supported in the subsurface bounded by only $v.$ It follows that any $a\in A$ and $c\in C$ are disjoint.
     \begin{figure}[h]
        \centering
        \begin{tikzpicture}
        \node[anchor = south west, inner sep = 0] at (0,0){\includegraphics[width=3in]{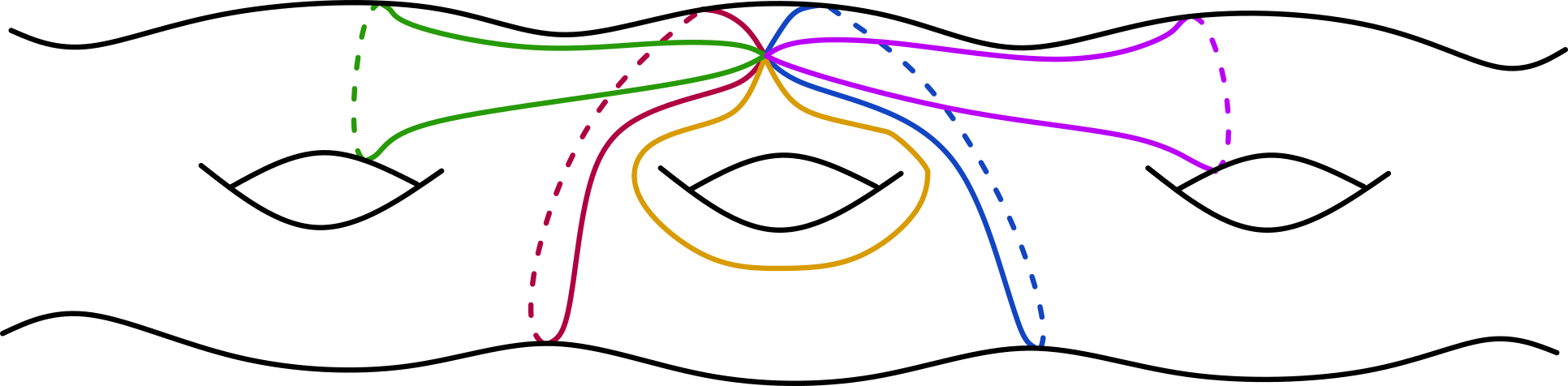}};
        \node at (3,0.7){$u$};
        \node at (5.1, 0.7){$v$};
        \node at (1.6, 1.6) {$a$};
        \node at (4, 0.4) {$b$};
        \node at (6.1, 1.6) {$c$};
        \end{tikzpicture}
         \caption{If $u$ and $v$ form a pants pair, we can find nonseparating curves $a\in A$, $b\in B$, and $c\in C$ such that any two form a pants pair. }
        \label{seppantsevenmorefig}
    \end{figure}
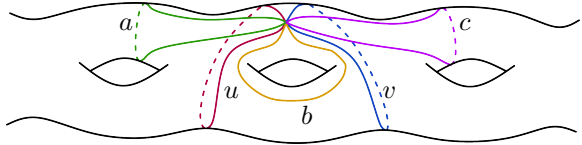
\end{proof}

Now that we have reduced Proposition~\ref{allpants} to the case of nonseparating pairs of curves, we are ready to prove Step 1.

\subsection{Step 1: Neighborhoods of nonseparating curves}\label{sec:neighborhoods}

In this section, we prove that neighborhoods of nonseparating curves are preserved by automorphisms of $\onefine(S_g).$ This is broken down into three main steps. In Step 1.1, we show that automorphisms  preserve adjacent homotopic nonseparating curves in $\onefine(S_g)$. In Step 1.2, we show that automorphisms preserve whether a curve is contained (in some sense) in the annulus bounded by two homotopic curves. In Step 1.3, we show that automorphisms preserve whether homotopic nonseparating curves form a pants pair. 

\subsubsection*{Step 1.1: Homotopic nonseparating curves}\label{sec:homnonsep}

The main result of this step is the following lemma. 

\begin{lemma}\label{nonsephomotaut}
   Let $u$ and $v$ be two adjacent nonseparating curves in $\onefine(S_g)$. Let $\varphi \in \Aut \onefine(S_g)$. Then $\varphi(u)$ is homotopic to $\varphi(v)$ if and only if $u$ is homotopic to $v$.  
\end{lemma}


Since homotopic curves are jointly separating, the first step in the proof of Lemma~\ref{nonsephomotaut} is to show that the set of jointly separating pairs of curves is preserved by automorphisms of $\onefine(S_g).$

\begin{lemma}\label{separatedns}
    Let $\varphi\in\Aut \onefine(S_g)$ for $g\geq 2.$ Let $u$ and $v$ be adjacent nonseparating curves in $S_g$ that do not form a torus pair. Then $u$ and $v$ are not jointly separating if and only if $\varphi(u)$ and $\varphi(v)$ are not jointly separating. 

    It follows that $u$ and $v$ are jointly separating if and only if $\varphi(u)$ and $\varphi(v)$ are jointly separating. 
\end{lemma}

\begin{proof}
    We prove the lemma by showing the following: $u$ and $v$ are not jointly separating if and only if there exists a separating curve $c$ such that $c$ separates $u$ and $v.$

     \pit{Suppose $u$ and $v$ are jointly separating} In this case, any separating curve $c$ in $\link(u,v)$ would have to lie in a single component of $S\setminus(u\cup v)$; otherwise, $c$ would intersect $u$ or $v$ at least twice. Therefore, any such $c$ does not separate $u$ and $v.$

    \pit{Suppose $u$ and $v$ are neither jointly separating nor a torus pair}  In this case, there is a separating curve that separates $u$ from $v$ as in Figure \ref{Thm2.3a_2}. To find such a curve, cut $S$ along $u$ and $v$ (retaining the boundaries), and take $c$ to be the separating curve that forms a (potentially pinched) pair of pants with the boundaries arising from $u.$
    \begin{figure}[h]
        \centering
        \begin{tikzpicture}
        \node[anchor = south west, inner sep = 0] at (0,0){\includegraphics[width=2in]{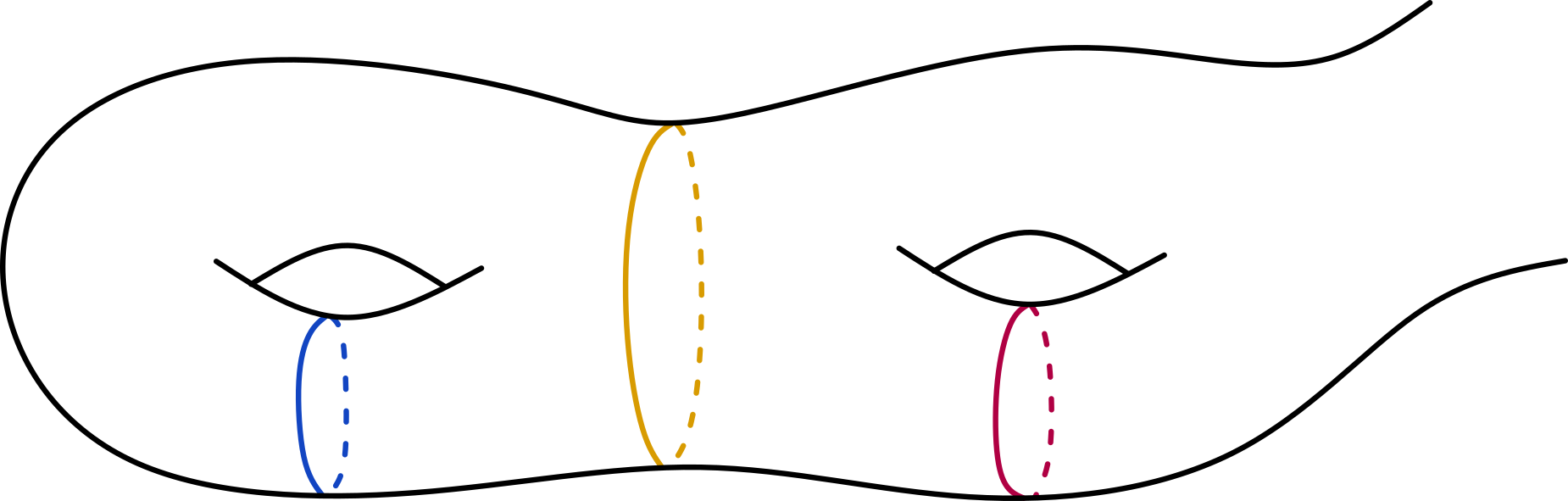}};
        \node at (0.8,0.3) {$u$};
        \node at (3.1, 0.3) {$v$};
        \node at (1.9, 1){$a$};
        \node[anchor = south west, inner sep = 0] at (6,0) {\includegraphics[width=2in]{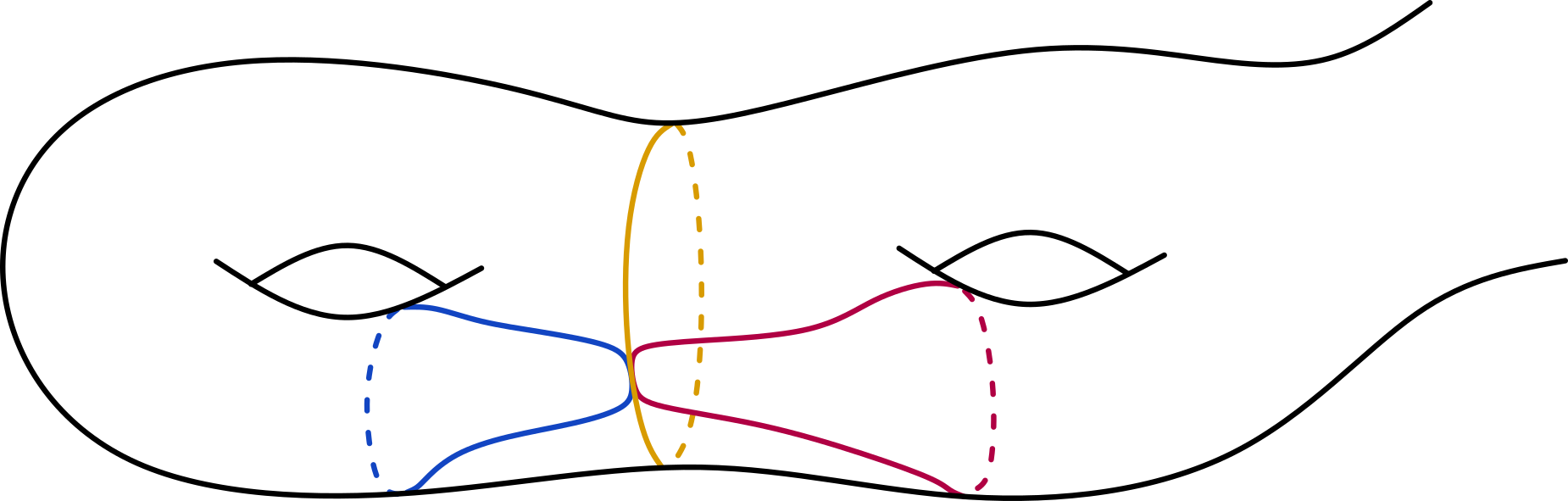}};
         \node at (7.0,0.3) {$u$};
        \node at (9.1, 0.3) {$v$};
        \node at (7.9, 1){$a$};
        \end{tikzpicture}
        \caption{Curves separating $u$ from $v$}
        \label{Thm2.3a_2}
    \end{figure}
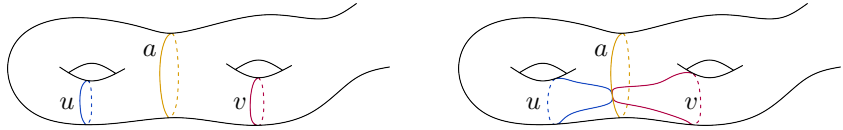
\end{proof}

A pair of adjacent homotopic curves is jointly separating, so we now further distinguish between homotopic jointly separating curves and non-homotopic jointly separating curves.

\begin{lemma}\label{homjointlysep}
     Let $u$ and $v$ be two adjacent jointly separating nonseparating curves in $\onefine(S)$. 
     Then, $u$ and $v$ are homotopic if and only if $\varphi(u)$ and $\varphi(v)$ are homotopic.
\end{lemma}

\begin{proof}
    By Lemma~\ref{seplinkaction}, it is enough to show that $u$ and $v$ are homotopic if and only if $\Q(u,v)$ is not a join.
    
    \pit{Suppose $u$ and $v$ are homotopic} Let $[a]$ and $[b]$ be distinct vertices of $\Q(u,v).$ Then, $a$ and $b$ are necessarily in the same connected component of $S\setminus (u\cup v).$ Let $c$ be a curve in that same component of $S\setminus(u\cup v)$ such that no curve isotopic to $c$ is disjoint from neither $a$ nor $b.$ Then, every curve in the isotopy class $[d]=T_{[c]}[a],$ the Dehn twist of $[a]$ about $[c],$ intersects every curve in the isotopy class of $a$ and every curve in the isotopy class of $b$ at least twice. Choose a representative $d$ of $T_{[c]}[a]$ disjoint from $u\cup v$. Then we have that $[d]$, a vertex of $\Q(u,v)$ is neighbors with neither $[a]$ nor $[b].$ Since this is true for any $[a]$ and $[b]$, we conclude that $\Q(u,v)$ is not  join.

    \pit{Suppose $u$ and $v$ are not homotopic} Let $[a]$ be a vertex of $Q^{\sep}(u,v)$ and $a\in \link(u,v)$ be an arbitrary representative. Then, $a$ is contained in one connected component of $S\setminus (u,v)$ except possibly for points of intersection with $u$ and $v$. By pushing $a$ off of itself in the direction away from $u$ and $v,$ we obtain another curve $a'$ with $[a']=[a].$

    We notice that no curve $a'$ homotopic to $a$ is contained in a different connected component of $S\setminus (u\cup v)$ from $a,$ since then $a$ and $a'$ would bound a (potentially pinched) annulus that must contain $u$ or $v,$ a contradiction.

    We now conclude that $\Q(u,v)$ is indeed a join, where the parts correspond to equivalence classes of curves contained in each of the components of $S\setminus (u\cup v)$ (potentially except for intersections with $u$ and $v$). 
\end{proof}

We now show that homotopic nonseparating pairs of curves are preserved by automorphisms of $\onefine(S_g)$.

\begin{proof}[Proof of Lemma~\ref{nonsephomotaut}]
    By Proposition~\ref{toruspairaut}, torus pairs are preserved by automorphisms, and thus the set of disjoint and pants pairs are preserved by automorphisms. Disjoint and pants pairs may be jointly separating or not jointly separating; by Lemma~\ref{separatedns}, being (not) jointly separating is preserved by automorphisms. Finally, jointly separating pairs may be either homotopic or not. Lemma~\ref{homjointlysep} ascertains that, within the set of jointly separating pairs, homotopic pairs are preserved by automorphisms of $\onefine(S_g).$ By combining these results, we conclude that nonseparating homotopic pairs are preserved by automorphisms.
\end{proof}

\subsubsection*{Step 1.2: Containment in an annulus}\label{sec:containinannulus}

In this section, we prove that automorphisms preserve whether a curve lies in the annulus bounded by two homotopic curves. Such a curve must be homotopic to the two boundary curves, and therefore all three curves must be adjacent in $\onefine(S_g).$ 

First, we define what it means for a curve $w$ to lie in the annulus bounded by adjacent homotopic curves $u$ and $v$ in $\onefine(S_g).$  Let $\ann_{\leq1}(u,v) \subseteq \link(u,v)$ denote the subgraph generated by curves $w$ supported on the (possibly pinched) annulus bounded by $u$ and $v$ such that $|w\cap(u\cup v)| \leq 1$ as in Figure \ref{annulus}.   Our main goal in this section is to prove the following result.

\begin{lemma}\label{lemma:annularmain}
    Let $u,$ $v,$ and $w$ be pairwise adjacent homotopic nonseparating curves in $\onefine(S_g)$ and $\varphi\in\Aut\onefine(S_g)$. Then $w\in\ann_{\leq1}(u,v)$ if and only if $\varphi(w)\in \ann_{\leq1}(\varphi(u),\varphi(v)).$
\end{lemma}

\begin{figure}[h]
\begin{center}
\begin{tikzpicture}
    \node[anchor = south west, inner sep = 0] at (0,0){\includegraphics[width=2in]{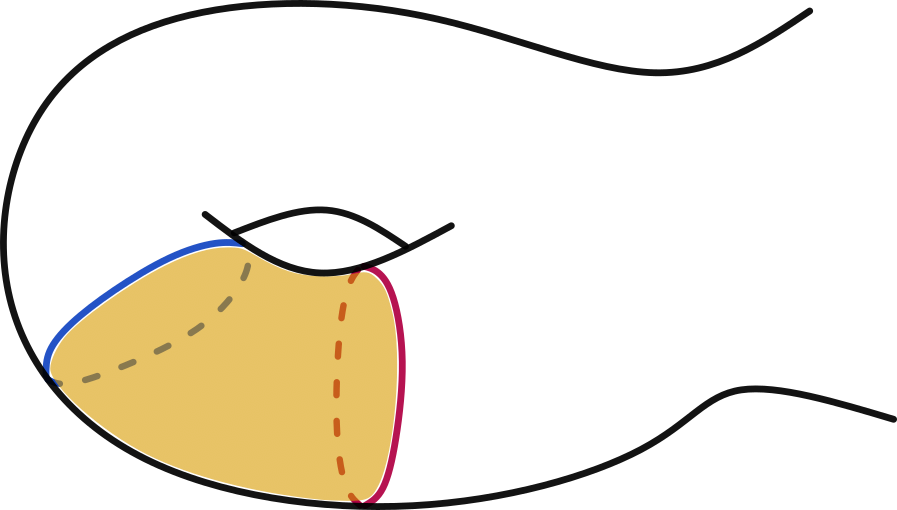}};
    \node at (2.5,1) {$u$};
    \node at (0.8,1.6) {$v$};
    \end{tikzpicture}
    \caption{Annulus formed by a pair of homotopic, nonseparating curves $u$ and $v$}\label{annulus}
\end{center}
\end{figure}

\begin{proof}
    Since elements of $\Aut\onefine(S_g)$ preserve separating links by Lemma \ref{sepdisjointhom}, it is enough to show that $w \in \ann_{\leq1}(u,v)$ if and only if $\link^{\text{sep}}(u,v)\subset \link^{\text{sep}}(w)$. 

    \pit{Suppose $w \in \ann_{\leq1}(u,v)$} Let $a\in \link^{\text{sep}}(u,v)$. Since $|w\cap (u,v)|\leq 1$ and $w$ is in the annulus bounded by $u$ and $v$ while $a$ is not, $|a\cap w|\leq 1.$ We conclude that $a\in\link^{\text{sep}}(w).$

    \pit{Suppose $w \not\in \ann_{\leq1}(u,v)$} Then, $w$ is either (1) in the annulus bounded by $u$ and $v$ and intersects both of them or (2) not in the annulus bounded by $u$ and $v.$

    Suppose $w$ is in case (1), and take any separating curve $a\in\link(u,v)$ disjoint from $u\cup v.$ We can then isotope $a$ to touch $u$ and $v$ at $u\cap w$ and $v\cap w,$ respectively. Thus $a\not\in \link(w).$

    Suppose $w$ is in case (2), and take any separating curve $a\in\link(u,v).$ We can then isotope $a$ to intersect $w$ at least twice, so $a\not\in\link(w)$.

    We conclude that $\link^{\text{sep}}(u,v)\not\subset \link^{\text{sep}}(w)$.
\end{proof}

\subsubsection*{Step 1.3: Homotopic nonseparating pants pairs}\label{sec:homnonseppants}

We now state and prove the result which allows us to distinguish pants pairs from disjoint pairs in the case that $u$ and $v$ are homotopic nonseparating curves. 

\begin{lemma}\label{homnonseppants}
    Let $u$ and $v$ be adjacent homotopic nonseparating curves in $S_g$ and $\varphi\in\Aut\onefine(S_g).$ Then, $(u,v)$ is a pants pair if and only if $(\varphi(u),\varphi(v))$ is a pants pair. It follows that $(u,v)$ is a disjoint pair if and only if $(\varphi(u),\varphi(v))$ is a disjoint pair.
\end{lemma}

The following lemma provides the combinatorial conditions necessary to prove Lemma~\ref{homnonseppants}.

\begin{figure}[h]
\begin{center}
    \begin{tikzpicture}
    \node[anchor = south west, inner sep = 0] at (0,0){\includegraphics[width=2in]{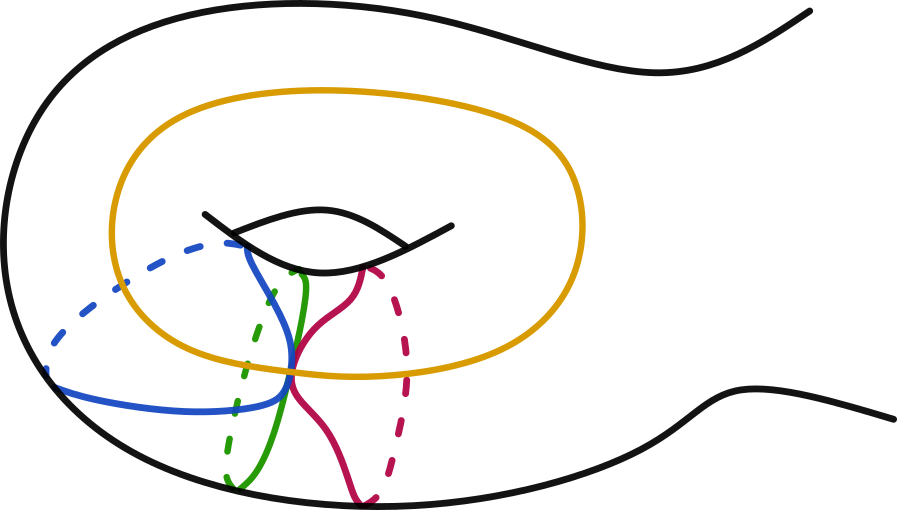}};
    \node at (2,0.5) {$u$};
    \node at (0.5, 1.3) {$v$};
    \node at (2,2.5) {$\gamma$};
    \node at (1.6,0.3) {$\delta$};
    \node[anchor = south west, inner sep = 0] at (6,0){\includegraphics[width=2in]{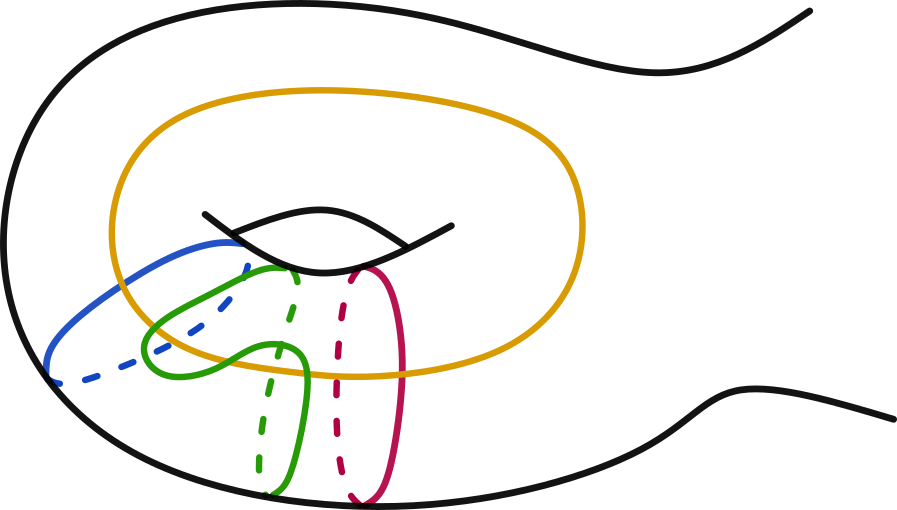}};
    \node at (8.2,0.5) {$u$};
    \node at (6.5, 1.3) {$v$};
    \node at (8,2.5) {$\gamma$};
    \node at (7.55,0.3) {$\delta$};
    \end{tikzpicture}
    \caption{We can distinguish homotopic curves that are disjoint from those that form a pants pair.}\label{touch1fig}
\end{center}
\end{figure}
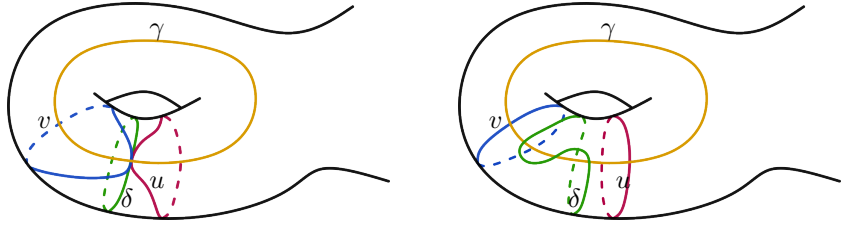

\begin{lemma}\label{nsdisjointhom}
    Let $u$ and $v$ be adjacent homotopic nonseparating curves in $\onefine(S_g)$. Then $(u,v)$ is a pants pair if and only if there exists a curve $\gamma \in \link(u,v)$ such that 1) $\gamma$ forms torus pairs with both $u$ and $v$ and 2) for any $\delta \in \ann_{\leq1}(u,v)$, $\gamma$ and $\delta$ are adjacent in $\onefine(S_g)$. Otherwise, $u$ and $v$ are disjoint.
\end{lemma}
\begin{proof}
    Suppose $(u,v)$ is a pants pair. Choose a curve $\gamma\in \link(u,v)$ that forms a torus pair with $u$ and $v$ and passes through their point of intersection. Then, $\gamma$ intersects every essential curve contained in $\ann_{\leq1}(u,v)$ exactly once, and is therefore adjacent to them. A schematic of this situation is pictured on the left in Figure~\ref{touch1fig}.

    Suppose $u$ and $v$ are disjoint. Choose any curve $\gamma\in\link(u,v)$ that forms a torus pair with both $u$ and $v$. Since an interval of $\gamma$ lies in the annulus bounded by $u$ and $v,$ we can choose a curve in $\ann_{\leq1}(u,v)$ that intersects $\gamma$ at least twice. An example of such a curve is shown on the right in Figure~\ref{touch1fig}.
\end{proof}

\begin{proof}[Proof of Lemma~\ref{homnonseppants}]
    The statement follows directly from Lemma~\ref{nsdisjointhom}, as torus pairs (Proposition~\ref{toruspairaut}) and being an element of $\ann_{\leq1}(u,v)$ (Lemma~\ref{lemma:annularmain}) are both preserved by automorphisms of $\onefine(S_g).$
\end{proof}

Lemma~\ref{homnonseppants} will allow us to conclude Proposition~\ref{allpants} in the case that $u$ and $v$ are homotopic.

\subsection{Step 2: Non-homotopic nonseparating curves: pants pairs vs. disjoint pairs}\label{sec:nonsepnonhompants}

The main result from this section is Lemma~\ref{nonsepnonhompants}. It proves Proposition~\ref{allpants} in the case that the edge $(u,v)$ consists of nonseparating, non-homotopic curves.

\begin{lemma}\label{nonsepnonhompants}
    Let $u$ and $v$ be adjacent non-homotopic nonseparating curves in $S_g.$ Then, for any $\varphi\in\Aut\onefine(S_g),$ we have that $(u,v)$ is a pants pair if and only if $(\varphi(u),\varphi(v))$ is a pants pair.  We therefore have that automorphisms of $\onefine(S_g)$ also preserve disjoint pairs of curves.
\end{lemma}

The main tool used to prove the above lemma is Lemma~\ref{nonseppants}, which provides a necessary and sufficient combinatorial condition for non-homotopic nonseparating curves to form a pants pair.

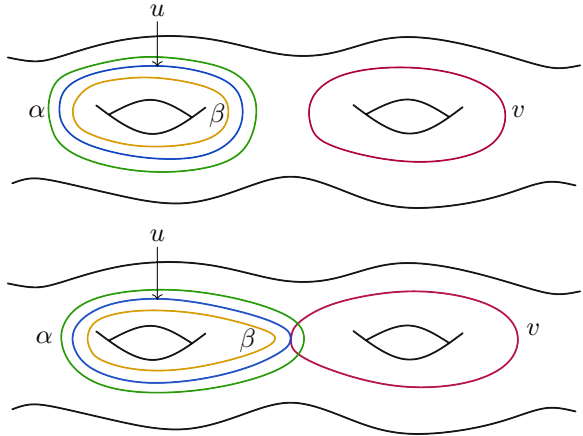
\begin{figure}[h]
\begin{center}
    \begin{tikzpicture}
    \node[anchor = south west, inner sep = 0] at (0,3){\includegraphics[width=3in]{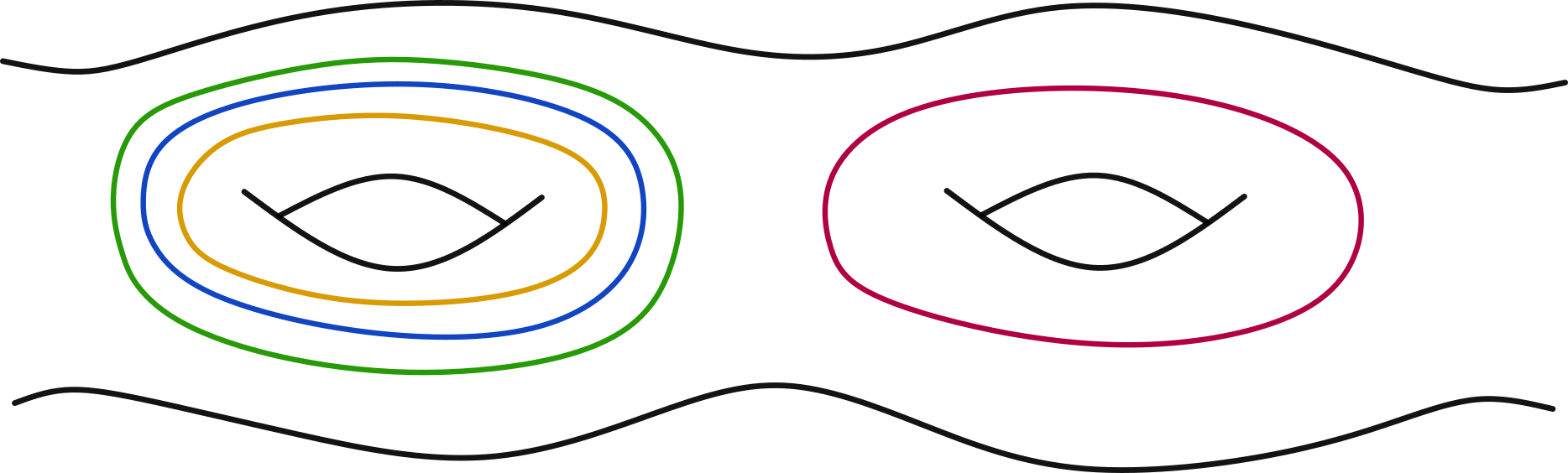}};
    \node at (0.4, 4.3) {$\alpha$};
    \node at (2.8, 4.25) {$\beta$};
    \node at (6.8,4.3) {$v$};
    \node at (2,5.65) {$u$};
    \draw [->, ultra thin] (2,5.5) -- (2,4.9);
    \node[anchor = south west, inner sep = 0] at (0,0){\includegraphics[width=3in]{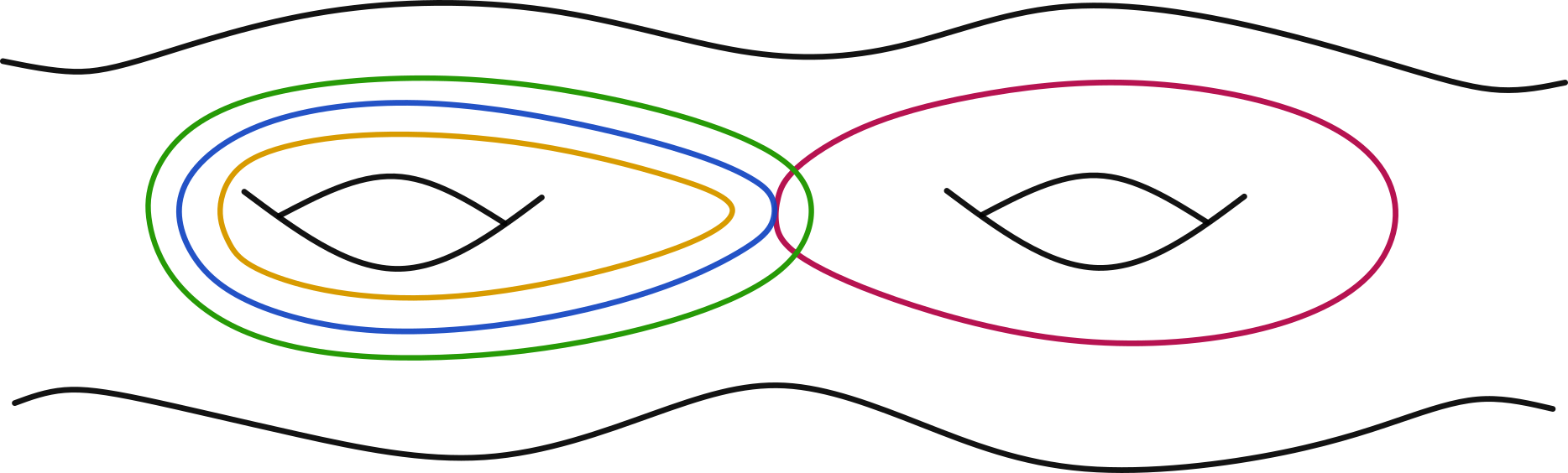}};
    \node at (0.5, 1.3) {$\alpha$};
    \node at (3.2, 1.25) {$\beta$};
    \node at (7,1.4) {$v$};
    \node at (2,2.65) {$u$};
    \draw [->, ultra thin] (2,2.5) -- (2,1.8);
    \end{tikzpicture}
    \caption{We can distinguish non-homotopic curves that are disjoint from those that form a pants pair.}\label{touch2fig}
\end{center}
\end{figure}

\begin{lemma}\label{nonseppants}
    Let $u$ and $v$ be adjacent non-homotopic nonseparating curves in $\onefine(S)$.  Suppose that $(u,v)$ is not a torus pair.  Then $u$ and $v$ are disjoint if and only if there exist adjacent curves $\alpha,\beta\in\link(u,v)$, such that
    \begin{enumerate}[label=(\arabic*)]
       \item $\alpha$, $u$, and  $\beta$ are homotopic,
        \item $\alpha$, $u$, and $\beta$ are disjoint, and
        \item $u \in \ann_{\leq1}(\alpha, \beta).$
    \end{enumerate}
    Otherwise, $(u, v)$ is a pants pair.
\end{lemma}

The main idea behind Lemma~\ref{nonseppants} is that $u$ and $v$ are disjoint if and only if there is a regular neighborhood of $u$ that is disjoint from $v.$  Figure \ref{touch2fig} gives a schematic of the proof.

\begin{proof}[Proof of Lemma~\ref{nonseppants}]
    Suppose first that $u$ and $v$ are disjoint. Then, there is an annular neighborhood of $u$ disjoint from $v.$ The boundary components of such a neighborhood are the desired curves $\alpha$ and $\beta.$
    
    Suppose now that such $\alpha$ and $\beta$ exist.  Let $N$ be the annulus with boundary components $\alpha$ and $\beta$.  By hypothesis, $u \subseteq \text{Int}(N)$. If $u \cap v \neq \emptyset$, then $v$ must intersect either $\alpha$ or $\beta$ in two places, which contradicts the assumption that $\alpha, \beta \in \link(u,v)$.  
\end{proof}

We are now ready to complete the main result of the Section \ref{sec:nonsepnonhompants}.

\begin{proof}[Proof of Lemma~\ref{nonsepnonhompants}]
Suppose that $u$ and $v$ are disjoint. Let $\alpha$ and $\beta$ be as in Lemma~\ref{nonseppants}. Then $\varphi$ preserves property (1) of $\alpha$ and $\beta$ by Lemma~\ref{nonsephomotaut}, property (2) by Lemma~\ref{homnonseppants}, and property (3) by Lemma~\ref{lemma:annularmain}. Hence $\varphi(\alpha)$ and $\varphi(\beta)$ realize $\varphi(u)$ and $\varphi(v)$ as being disjoint by Lemma \ref{nonseppants}, so the result follows.
\end{proof}

We are now ready to complete the main result of Section \ref{pantspairssection}.

\begin{proof}[Proof of Proposition~\ref{allpants}]
    Let $u$ and $ v$ be adjacent curves in $\onefine(S_g).$ We must show that if $(u,v)$ is a pants pair, then for any $\varphi\in \Aut\onefine(S_g),$ $(\varphi(u),\varphi(v))$ is also a pants pair. By Lemma \ref{seppantslemma}, we may assume that $u$ and $v$ are nonseparating. 
 We prove the proposition with casework.

    \pit{Case 1} If $u$ and $v$ are homotopic, nonseparating, and form a pants pair, Lemma~\ref{homnonseppants} asserts that $\varphi(u)$ and $\varphi(u)$ form a pants pair.

    \pit{Case 2} If $u$ and $v$ are non-homotopic, nonseparating, and form a pants pair, Lemma~\ref{nonsepnonhompants} asserts that $\varphi(u)$ and $\varphi(u)$ form a pants pair.
\end{proof}

\section{Torus Case}\label{torussection}

In this section, we approach the case where the surface $S_g$ has genus 1, and is therefore a torus. We will denote our surface by $T$ to avoid ambiguity. Our previous tools do not apply in the torus case because there are no essential separating curves on a torus. 

Proposition~\ref{toruscase} is the main result of this section.

\begin{proposition}\label{toruscase}
    Let $T$ be a torus. Then the natural map
    \[\Phi:\Homeo(T)\to \Aut\onefine(T)\]
    is an isomorphism.
\end{proposition}

We prove Proposition~\ref{toruscase} in two steps. First, in Section~\ref{sec:torustorusvselse}, we use the proof method of Le~Roux--Wolff \cite{transaut} to show that torus pairs are preserved by automorphisms of $\onefine(T)$. Then, in Section~\ref{sec:toruspantsvsdisjoint}, we show that pants pairs (and therefore disjoint pairs) are preserved by automorphisms of $\onefine(T).$ 

\subsection{Torus pairs vs. non-torus pairs}\label{sec:torustorusvselse}

The goal of this section is to prove Proposition~\ref{prop:torustorusvselse}.

\begin{proposition}\label{prop:torustorusvselse}
    Let $u$ and $v$ be adjacent curves in $\onefine(T)$ and  $\varphi\in\Aut\onefine(T).$ Then $(u,v)$ is a torus pair if and only if $(\varphi(u),\varphi(v))$ is a torus pair.
\end{proposition} 

We prove this proposition by building on the work of Le~Roux--Wolff. We begin by introducing the relevant results and explaining how we adapt their statements and proofs to apply in the case of the fine 1-curve graph on the torus. 

For any nonspherical, possibly non-orientable and possibly noncompact surface $S$, Le~Roux--Wolff study the graph $\onetrans(S)$, which has a vertex for every essential, simple, closed curve and two vertices are connected by an edge if they are either disjoint or they have one topologically transverse intersection point. (In the language of our paper, the edges correspond to disjoint pairs and torus pairs.) Le~Roux--Wolff show that $\Homeo(S)\cong \Aut\onetrans(S)$ via the natural homomorphism \cite{transaut}. 

To do this, they categorize all the possible curve configurations. A \textit{clique} is a collection of pairwise adjacent vertices in a graph. In particular, a clique $(a,b,c)$ is of type \textit{necklace} if $(a,b),\ (b,c),$ and $(c,a)$ are all torus pairs and do not have a common intersection point. Then they make the assertion that in $\onetrans(S),$
\begin{enumerate}[label=(\arabic*)]
    \item a clique $(a,b,c)$ is of type necklace if and only if there exists a finite set of (at most 8) vertices such that every $d\in\link(a,b,c)$ is adjacent to at least one vertex in the finite set,
    \item adjacent vertices $a, b$ are disjoint if and only if there is no $c\in\link(a,b)$ such that $(a,b,c)$ is of type necklace, and
    \item $a, b$ form a torus pair if and only if $a$ and $b$ are not disjoint.
\end{enumerate}

In the case of the fine 1-curve graph of the torus, these properties still hold with an adjustment to the last two:
\begin{enumerate}[label=(\arabic*)]
    \item [($2'$)] adjacent vertices $a, b$ are disjoint or a pants pair if and only if there is no $c\in\link(a,b)$ such that $(a,b,c)$ is of type necklace and
    \item  [($3'$)] adjacent vertices $a, b$ are a torus pair if and only if $a$ and $b$ are neither disjoint nor a pants pair.
\end{enumerate}

In fact, (1) is the main property we must verify en route to Proposition~\ref{prop:torustorusvselse}. It is stated in the following lemma.

\begin{lemma}\label{property1}
    A clique $(a,b,c)$ in $\onefine(T)$ is of type necklace if and only if there exists a finite set of (at most 8) vertices such that every $d\in\link(a,b,c)$ is adjacent to at least one vertex in the finite set.
\end{lemma}

The forward direction of Lemma~\ref{property1} is proven by Le~Roux--Wolff and applies without adaptations. The key to the backward direction is an adaptation to Lemma 2.5 of Le~Roux--Wolff, which we give as follows.

\begin{lemma}[Adaptation of Lemma 2.5 of Le~Roux--Wolff]\label{lerouxwolf25}
    Let $(a,b,c)$ be a clique not of type necklace. Then, there exists $d\in\link(a,b,c)$ such that $d$ intersects every component of $T\setminus \{a,b,c\}.$
\end{lemma}

This proof is done via casework on the possible arrangements of cliques. To aid in categorizing arrangements of cliques, we use an updated version of the notation of Le~Roux--Wolff: for a clique $(a,b,c)$, we record pairwise intersection types by a triple $(\cdot,\cdot,\cdot)$, up to permutation. If two curves are disjoint, we signify this by a 0; if they form a torus pair, by a 1; and if they form a pants pair, by a P. 

\begin{proof}[Proof of Lemma~\ref{lerouxwolf25}]
    The cases (1,1,1), (1,1,0), and (0,0,0) are all accounted for in Le~Roux--Wolff; (1,0,0) does not apply in the torus case. The remainder of the cases arise from replacing some 0's with P's.

    \pit{Case 1: (1,1,P)} In this case, we have two homotopic, touching curves and a third curve that forms a torus pair with them. There are two subcases: whether the three curves intersect at one point or at 3 distinct points.

    \textit{Case 1a: one point of intersection.} To create a fourth curve adjacent to all three that intersects all components of $T\setminus\{a,b,c\}$, we push the third curve off of itself. A schematic of this configuration is shown in Figure \ref{figure1a}. 

\begin{figure}[h]
    \begin{center}
        \begin{tikzpicture}
        \node[anchor = south west, inner sep = 0] at (0,0){\includegraphics[width=2in]{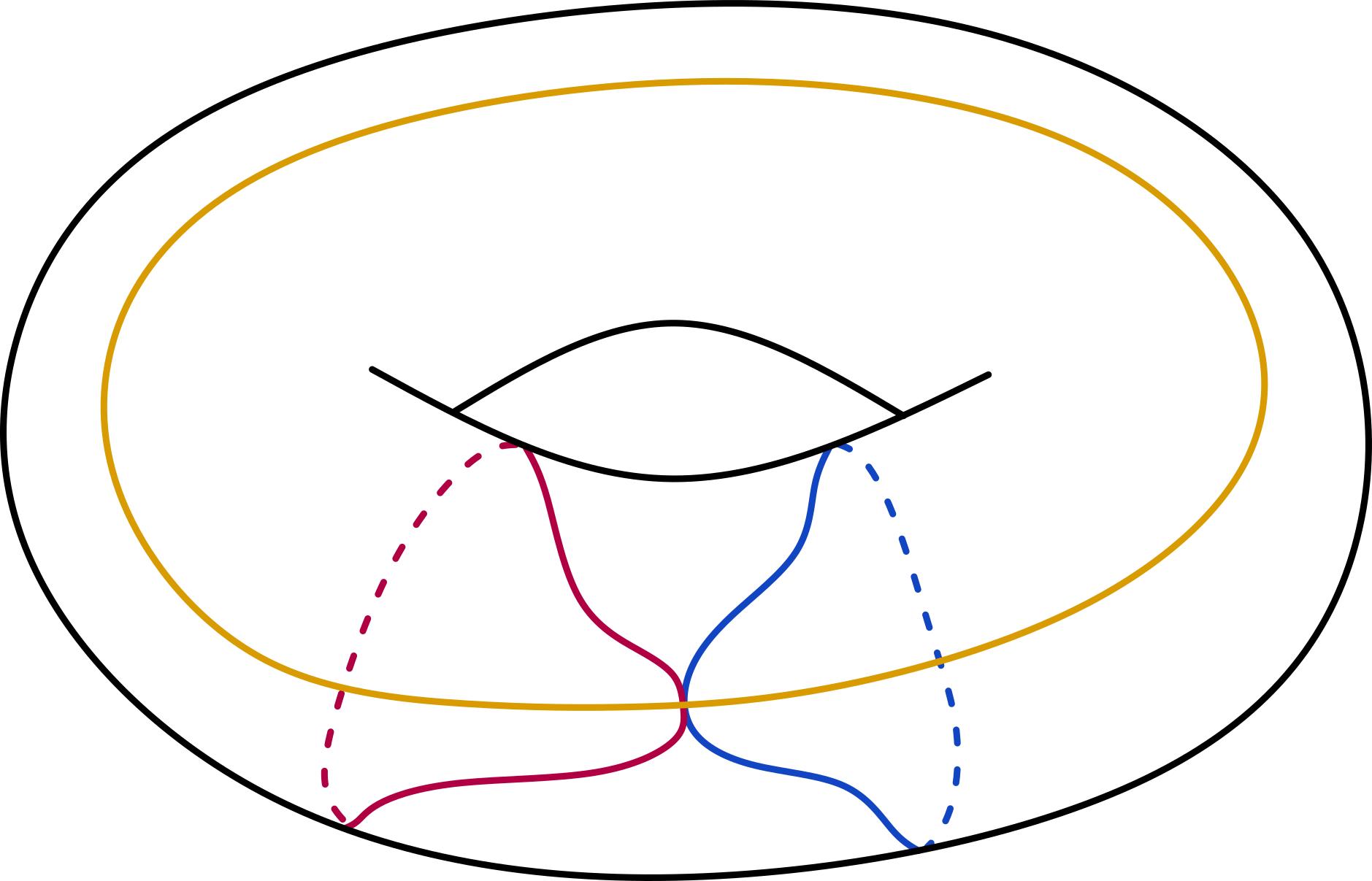}};
        \node at (2,1) {$a$};
        \node at (3, 1) {$b$};
        \node at (0.7,2) {$c$};
         \node[anchor = south west, inner sep = 0] at (6,0){\includegraphics[width=2in]{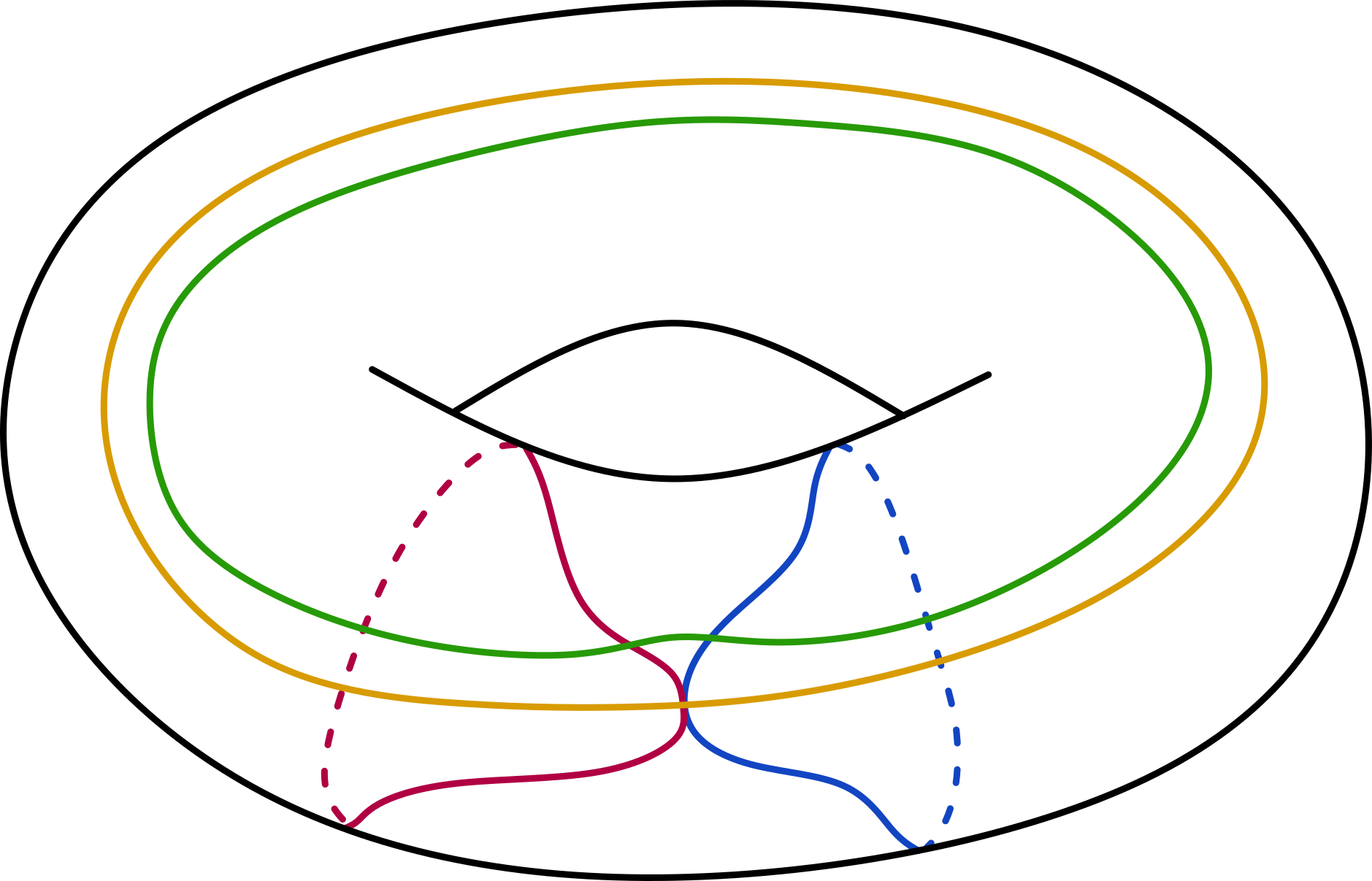}};
         \node at (6.8,2) {$d$};
        \end{tikzpicture}
        \caption{A schematic of Case 1a of Lemma~\ref{lerouxwolf25}}\label{figure1a}
    \end{center}
    \end{figure}
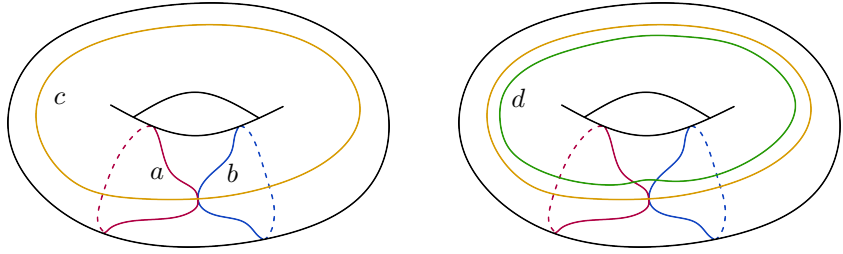

    \textit{Case 1b: three points of intersection.} The difficulty here is that $T\setminus\{a,b,c\}$ has 3 connected components, so pushing the third curve will no longer work. In particular, a curve that satisfies the conditions of the lemma must form a torus pair with each of $a,$ $b,$ and $c.$ A schematic of this configuration is shown in Figure \ref{figure1b}, along with a curve $d$ that satisfies the condition of the lemma.
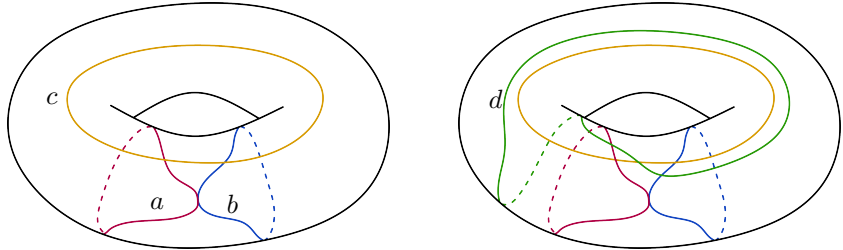
\begin{figure}[h]
    \begin{center}
    \begin{tikzpicture}
        \node[anchor = south west, inner sep = 0] at (0,0){\includegraphics[width=2in]{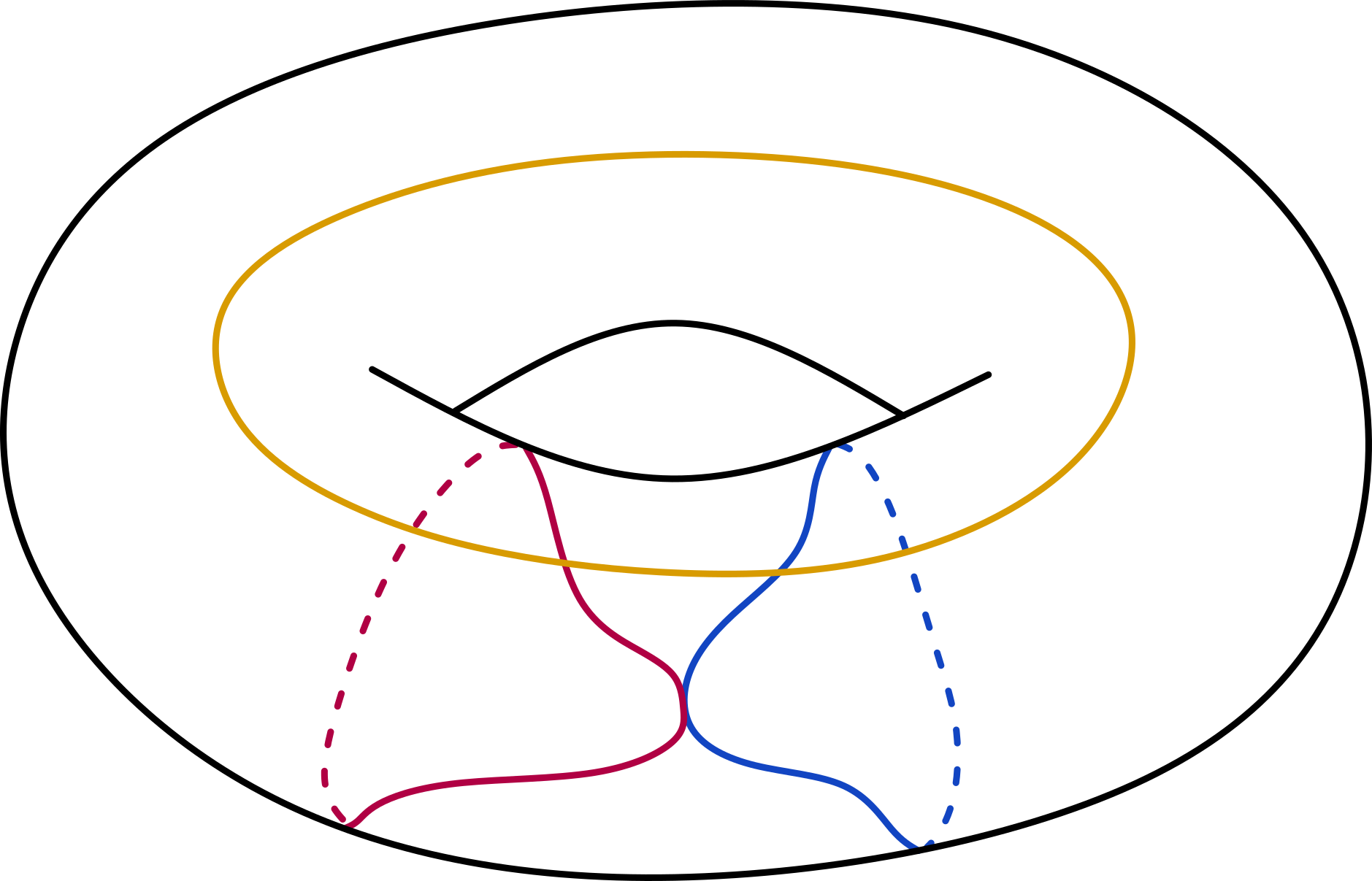}};
        \node at (2,0.6) {$a$};
        \node at (3, 0.6) {$b$};
        \node at (0.6,2) {$c$};
         \node[anchor = south west, inner sep = 0] at (6,0){\includegraphics[width=2in]{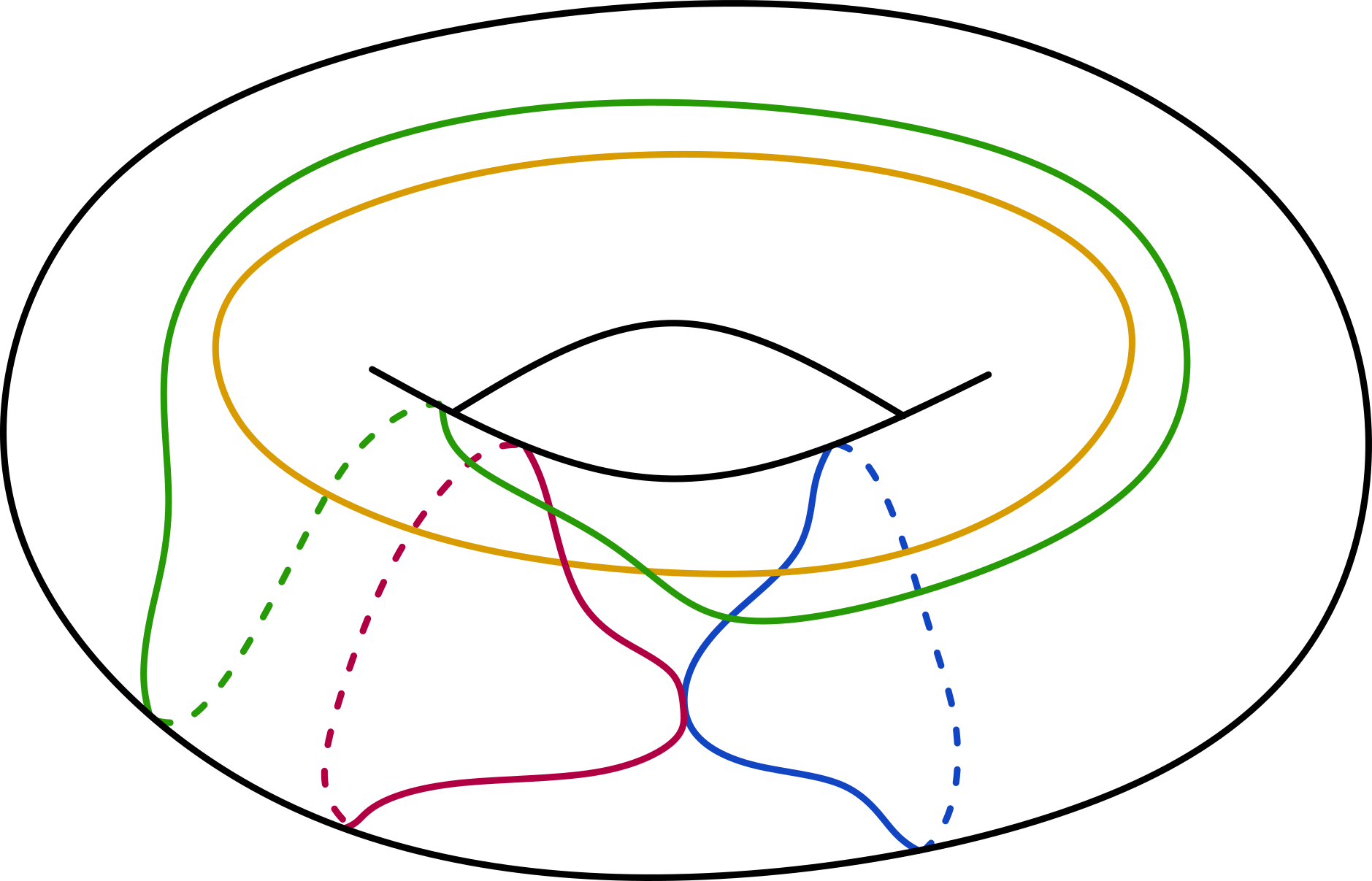}};
         \node at (6.5,2) {$d$};
        \end{tikzpicture}
        \caption{A schematic of Case 1b of Lemma~\ref{lerouxwolf25}}\label{figure1b}
    \end{center}
    \end{figure}

    \pit{Case 2: variants of (0,0,0): (P,0,0), (P,P,0), and (P,P,P)} In all of these configurations, all of $a,b,c$ are homotopic, and thus form (possibly pinched) annuli. A curve $d$ transverse to $a,b,$ and $c$ that does not cross any of the touching points satisfies the hypotheses of the lemma. A schematic of these configurations, along with a curve $d$ that satisfies the hypotheses of the lemma, is shown in Figure \ref{figurecase2}.
\end{proof}

\begin{figure}[h]
\begin{center}
    \begin{tabular}{|c|c|c|}
        \hline
        \raisebox{-.2em}{\textbf{(P,0,0)}} & \raisebox{-.2em}{\textbf{(P,P,0)}} & \raisebox{-.2em}{\textbf{(P,P,P)}} \\[.5em]
        \hline
        & & \\[-.8em]
        \raisebox{-.3em}{\includegraphics[width=1.15in]{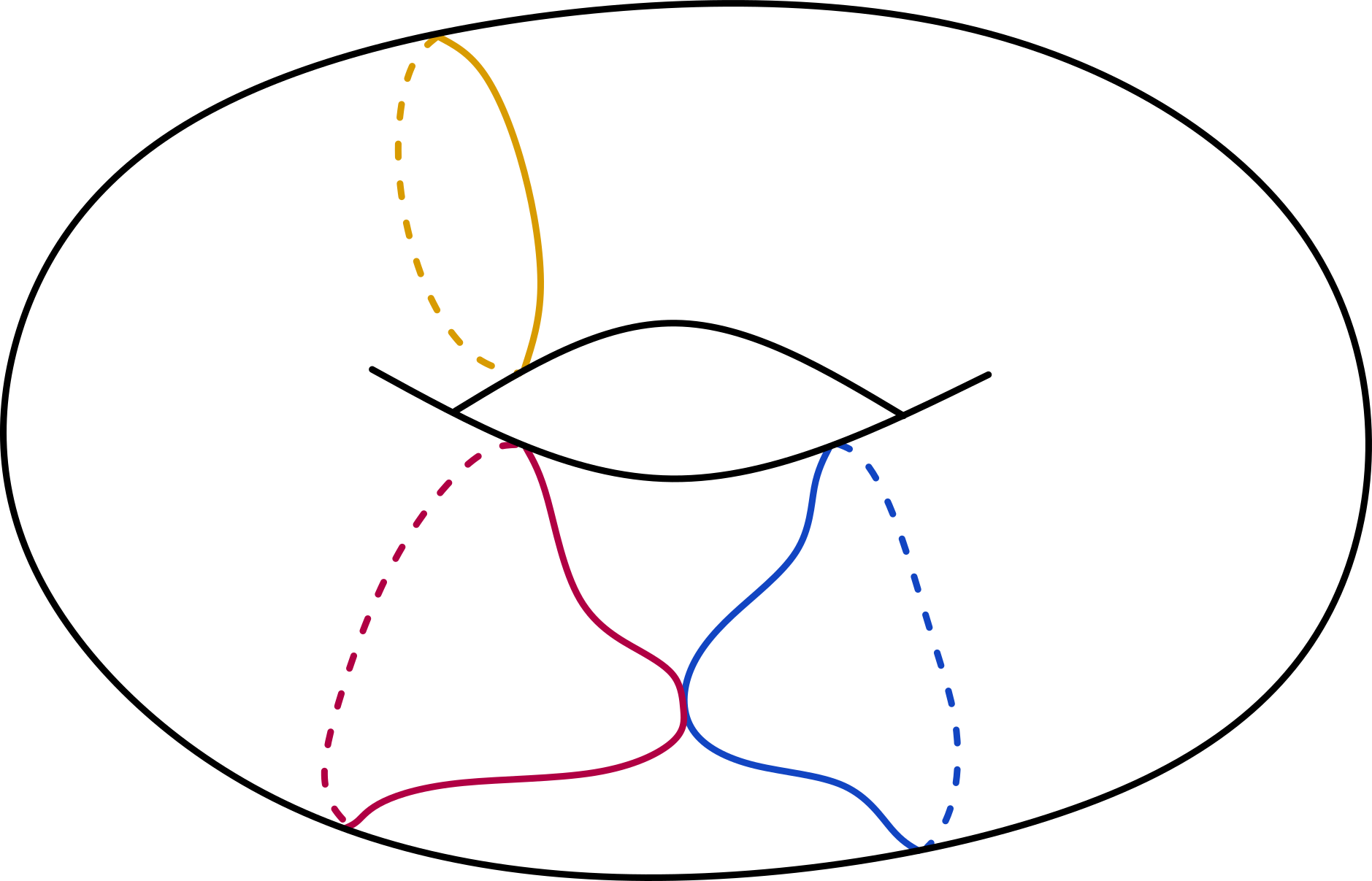}} & \raisebox{-.2em}{\includegraphics[width=1.15in]{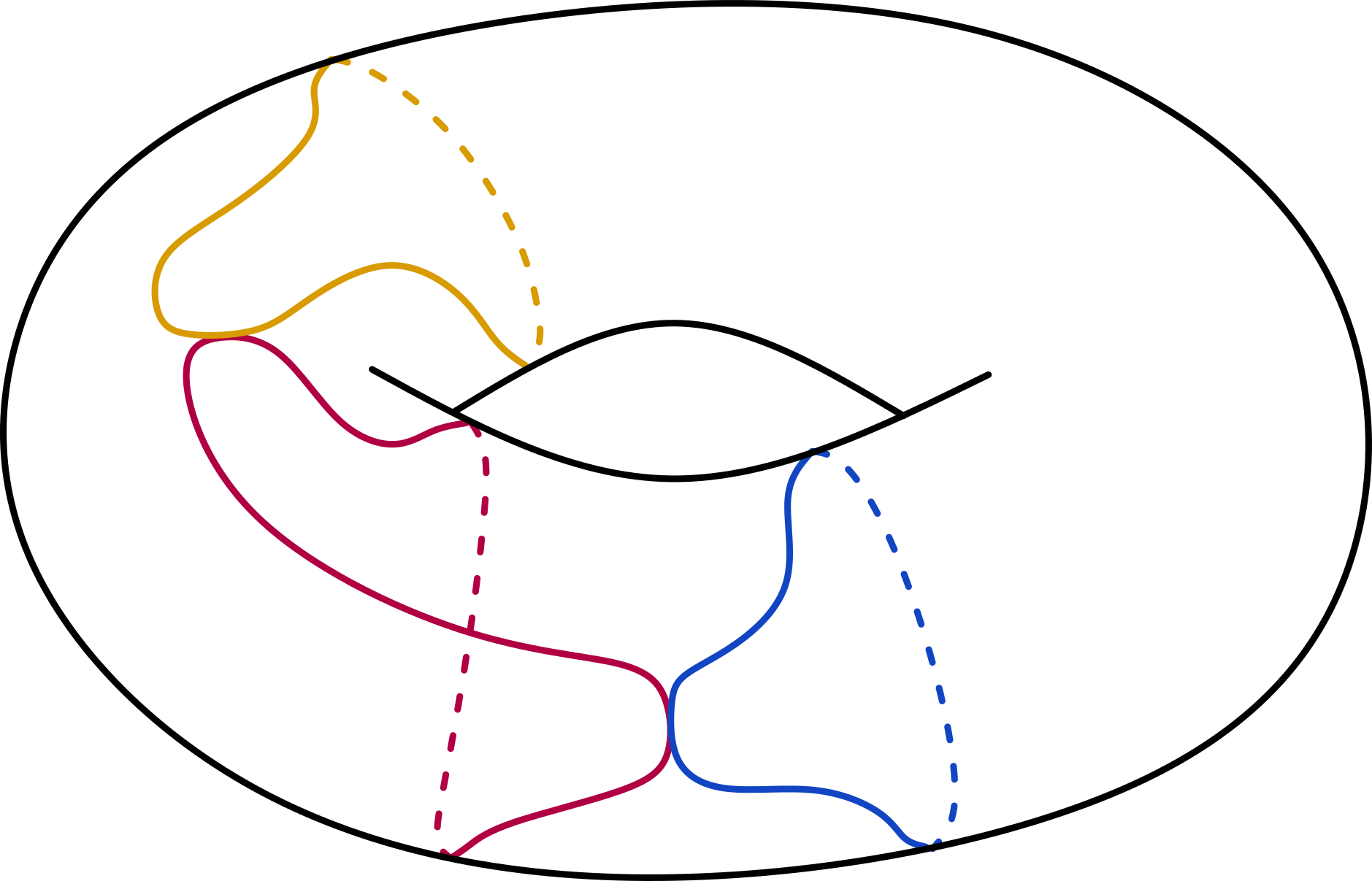}} & \raisebox{-.2em}{\includegraphics[width=1.15in]{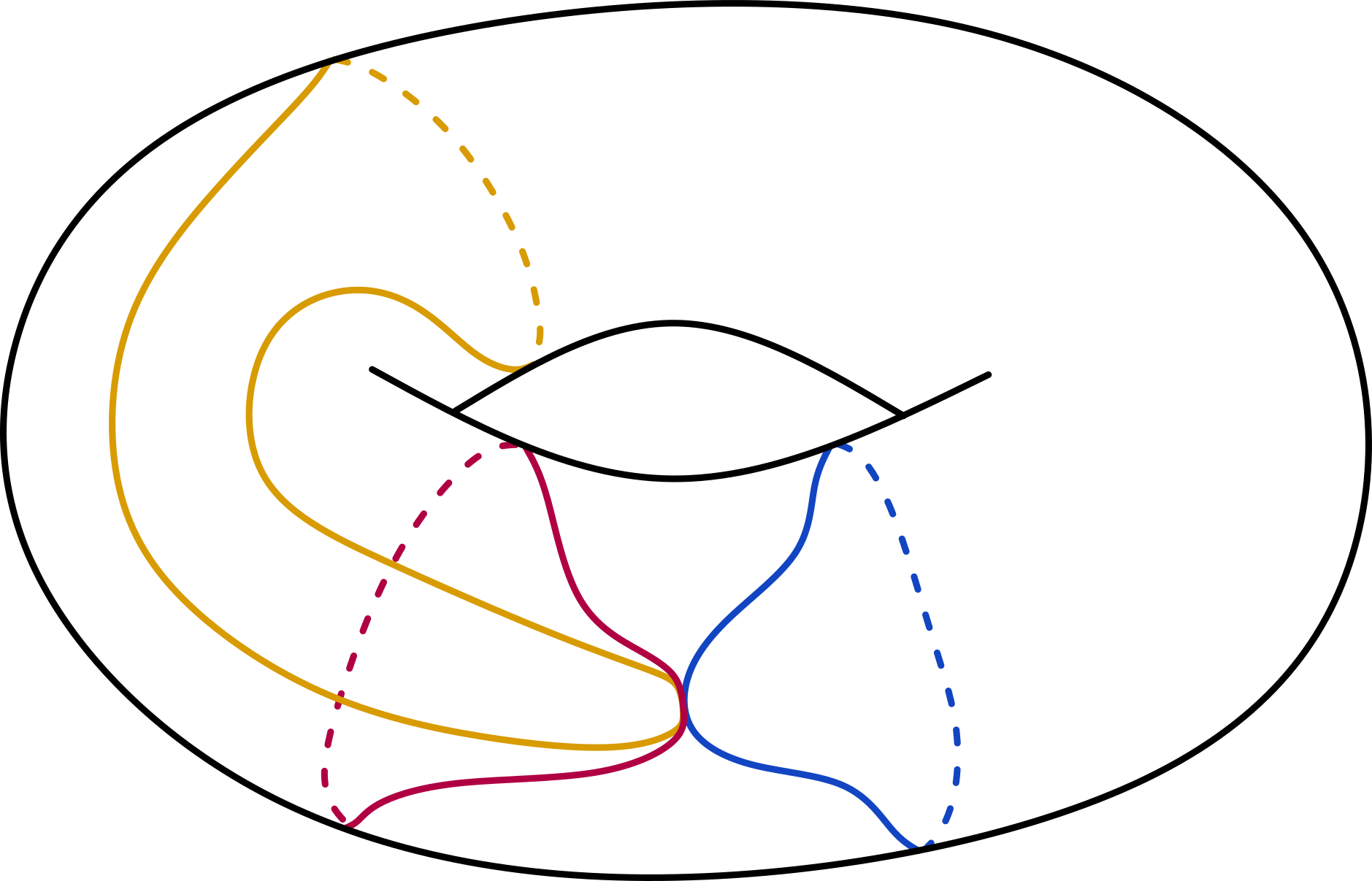}} \hspace{.05em} \raisebox{-.2em}{\includegraphics[width=1.15in]{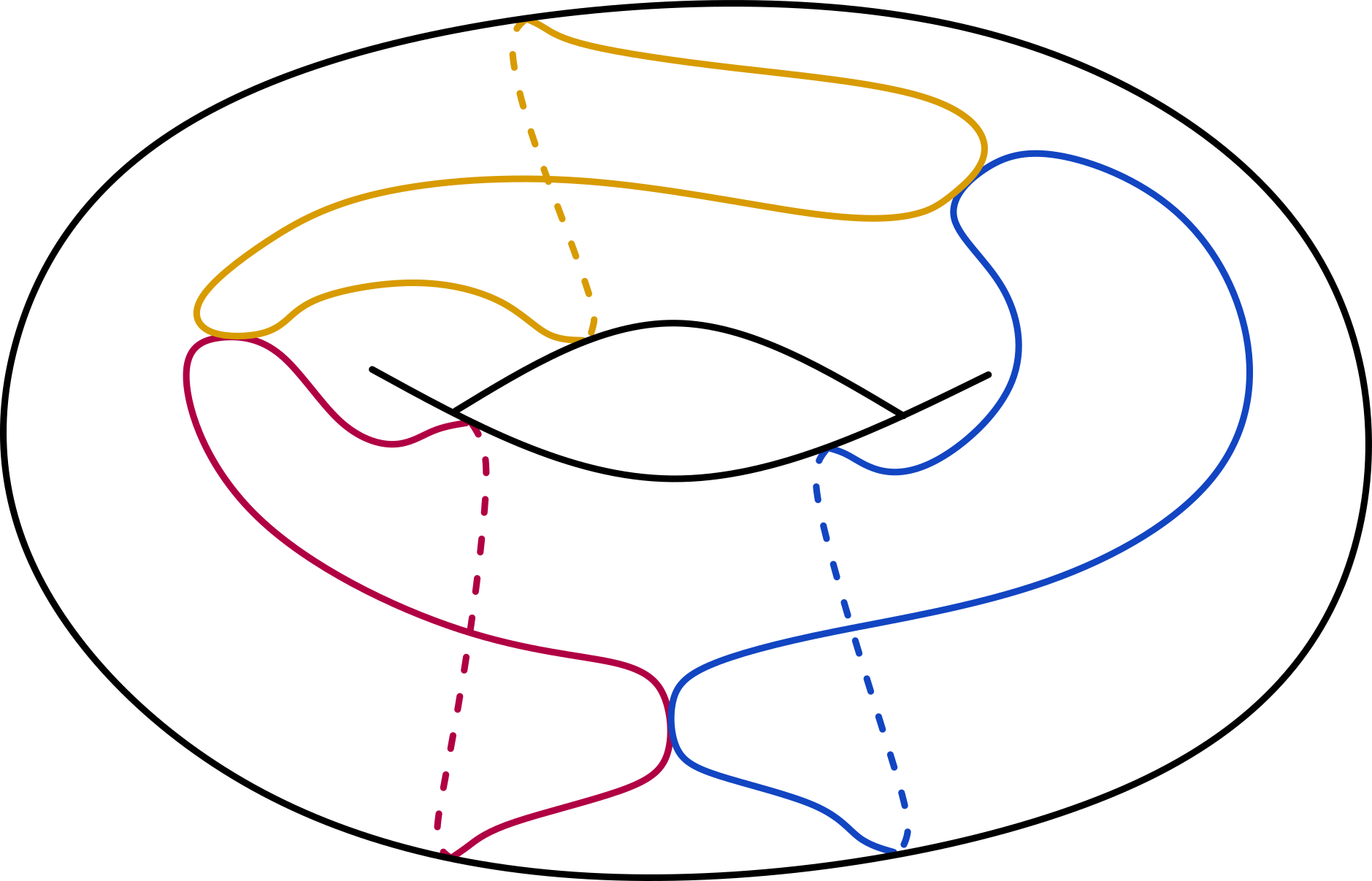}} \\[.3em]
        $\downarrow$ & $\downarrow$ & $\downarrow$ \hspace{1in} $\downarrow$ \\
        & & \\[-.9em]
        \raisebox{-.2em}{\includegraphics[width=1.15in]{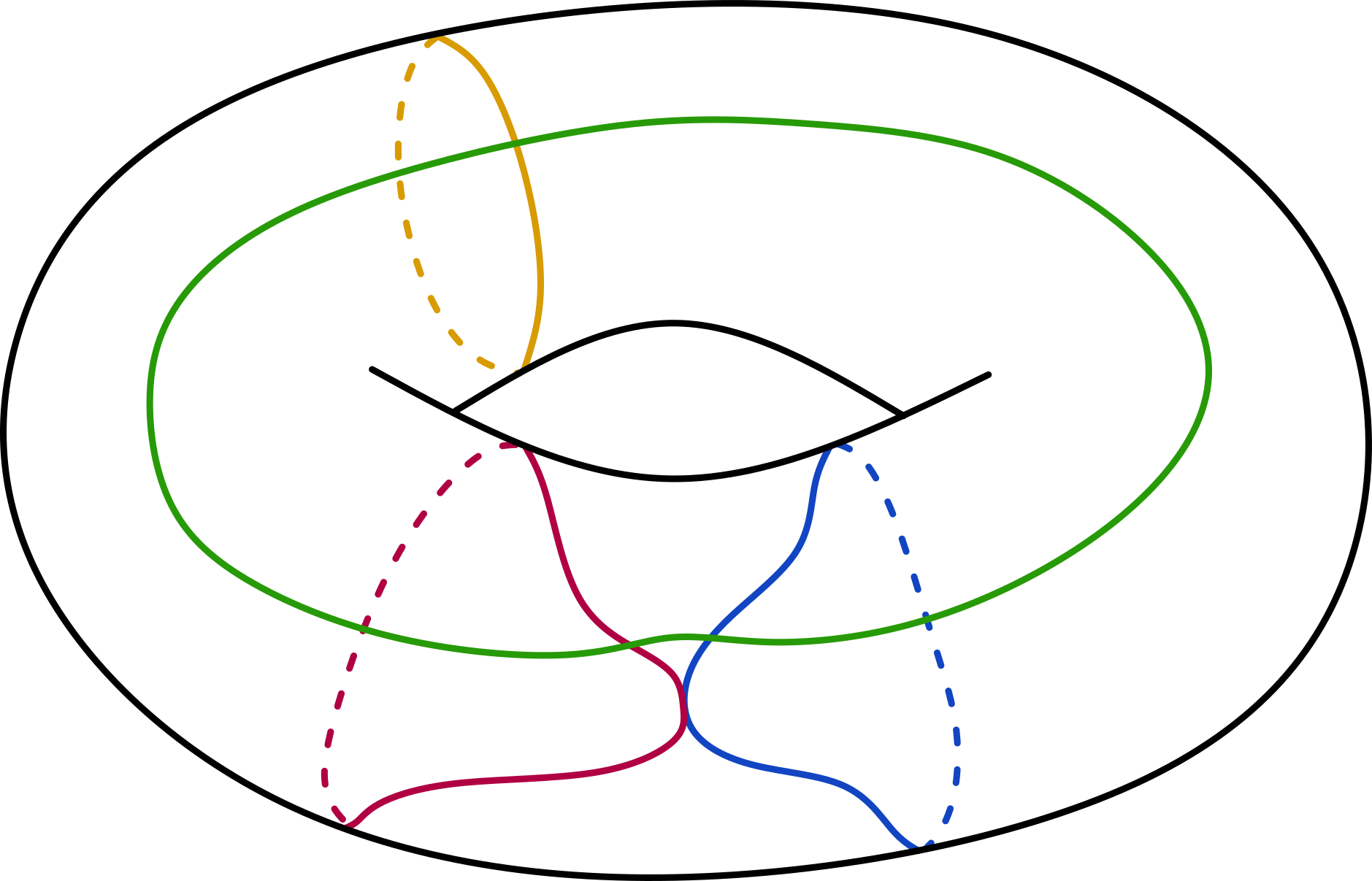}} & \raisebox{-.2em}{\includegraphics[width=1.15in]{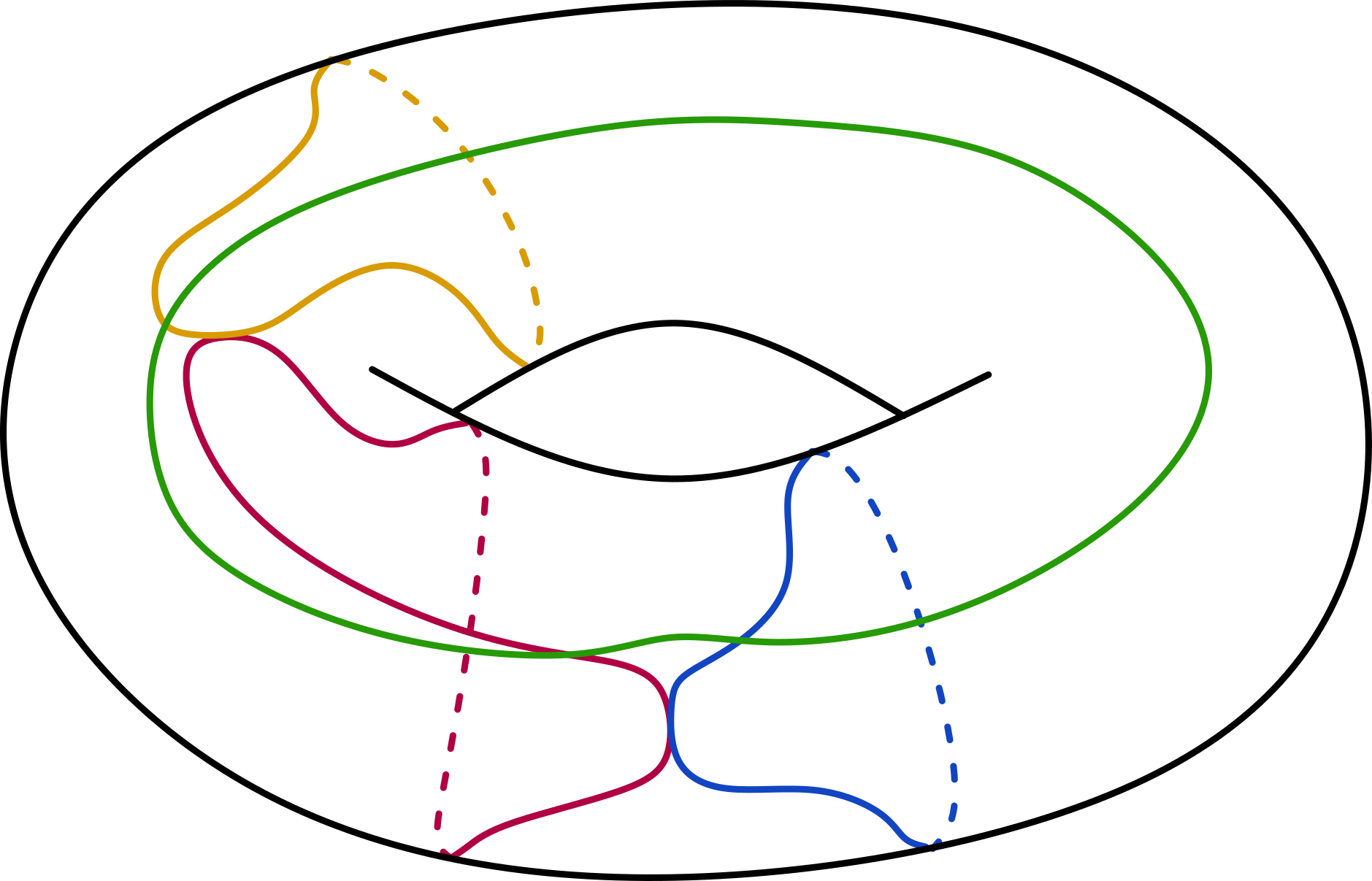}} & \raisebox{-.2em}{\includegraphics[width=1.15in]{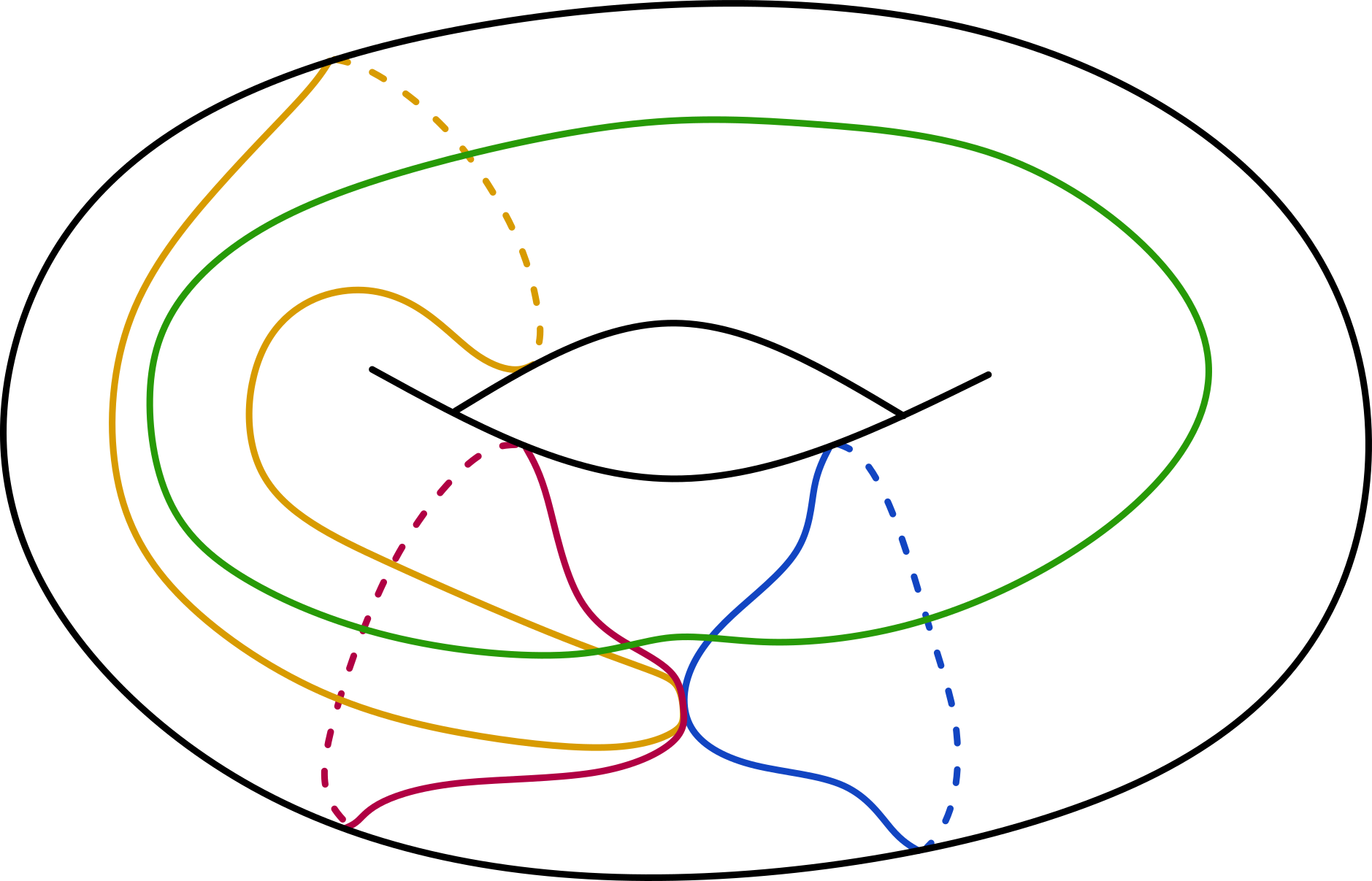}} \hspace{.05em} \raisebox{-.2em}{\includegraphics[width=1.15in]{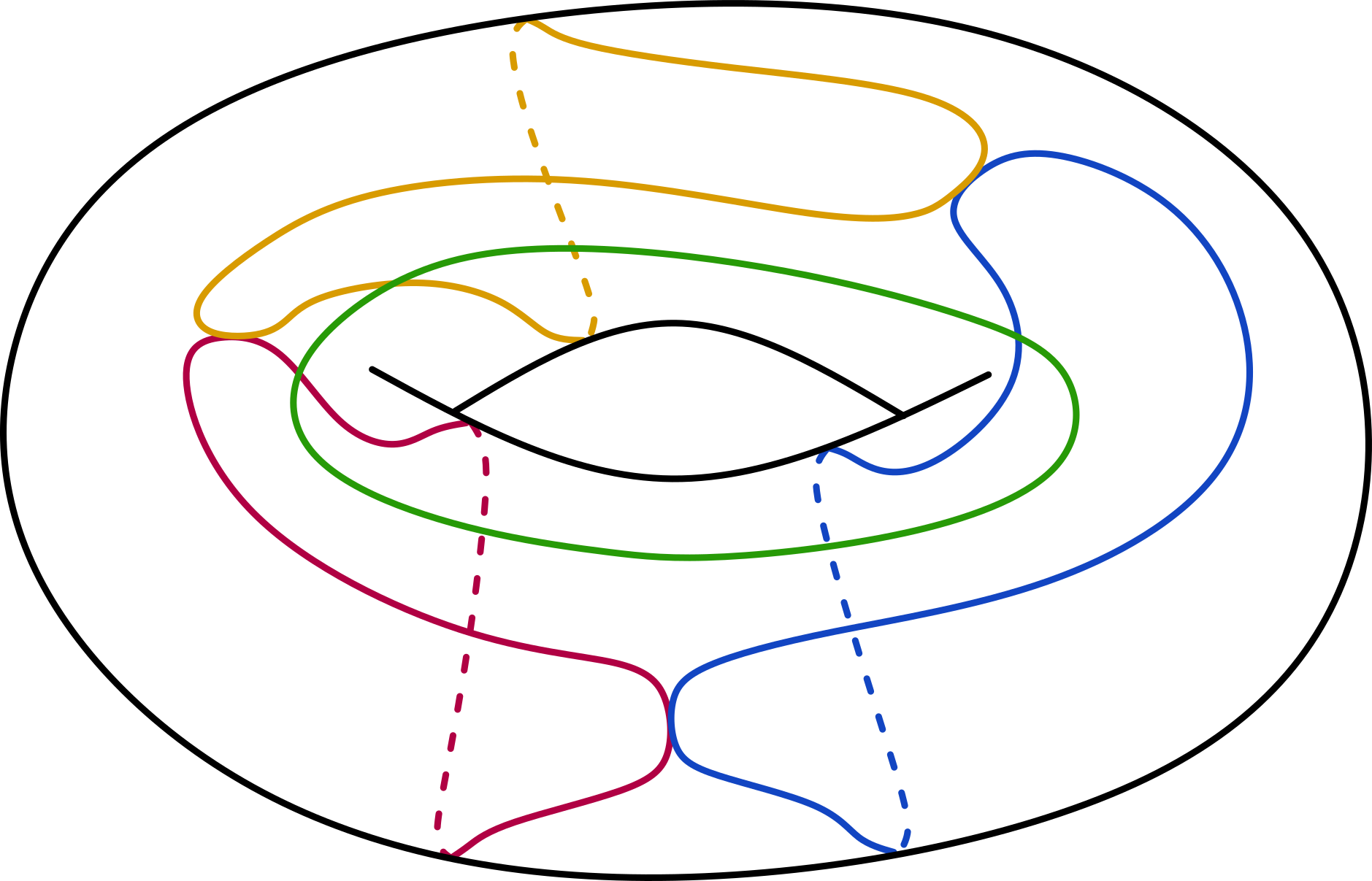}} \\[.3em]
        \hline
    \end{tabular}
    \caption{A schematic of Case 2 of Lemma~\ref{lerouxwolf25}}\label{figurecase2}
\end{center}
\end{figure}

Then, Lemma 2.6 of Le~Roux--Wolff holds with minor adaptations and an identical proof; here, it appears as Lemma~\ref{lerouxwolff26}. It provides the converse to Lemma~\ref{property1}.

\begin{lemma}\label{lerouxwolff26}
    Let $(a,b,c)$ be a clique in $\onefine(T)$ not of type necklace and $\{\alpha_1,\ldots,\alpha_j\}$ be a finite collection of curves distinct from $a,\ b,$ and $c.$ Then, there is a vertex $d\in\link(a,b,c)$ that is not adjacent to any $\alpha_i.$
\end{lemma}

The main idea of the proof is to start with a curve $d$ given by Lemma~\ref{lerouxwolf25} and isotope it in $T\setminus\{a,b,c\}$ to intersect each $\alpha_j$ arbitrarily many times.

We now have all the tools we need to prove Lemma~\ref{property1}.

\begin{proof}[Proof of Lemma~\ref{property1}]
    The forward direction is given by Lemma 2.8 of Le~Roux--Wolff. The proof applies because all edge relations in $\onetrans(T)$ still exist in $\onefine(T).$
    
    The reverse direction follows from Lemma~\ref{lerouxwolff26}.
\end{proof}

It remains to show properties ($2'$) and ($3'$) above. The following lemma is similar to Corollary 2.4 of Le~Roux--Wolff.

\begin{lemma}\label{properties2and3}
    In $\onefine(T),$ the following properties hold.
    \begin{enumerate}[label=(\arabic*)]
        \item [($2'$)] Adjacent vertices $a$ and  $b$ are disjoint or a pants pair if and only if there is no $c\in\link(a,b)$ such that $(a,b,c)$ is of type necklace and
        \item  [($3'$)] adjacent vertices $a$ and $b$ are a torus pair if and only if $a$ and $b$ are neither disjoint nor a pants pair.
    \end{enumerate}
\end{lemma}

\begin{proof}
    If adjacent curves $a$ and $b$ are disjoint or form a pants pair, then by definition, there is no third curve $c$ such that $(a,b,c)$ is a necklace clique. 
    
    Alternately, if $a$ and $b$ are a torus pair, then up to homeomorphism, they are the (1,0) and (0,1) curves on the torus. These curves, along with some (1,1) curve, form a necklace clique. A schematic of this configuration appears in Figure \ref{torusfigure}.

    Property ($3'$) follows from the definitions of torus pairs.
\end{proof}

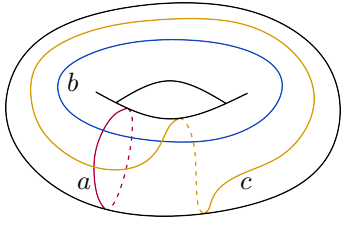
\begin{figure}[h]
\begin{center}
    \begin{tikzpicture}
    \node[anchor = south west, inner sep = 0] at (0,0){\includegraphics[width=1.75in]{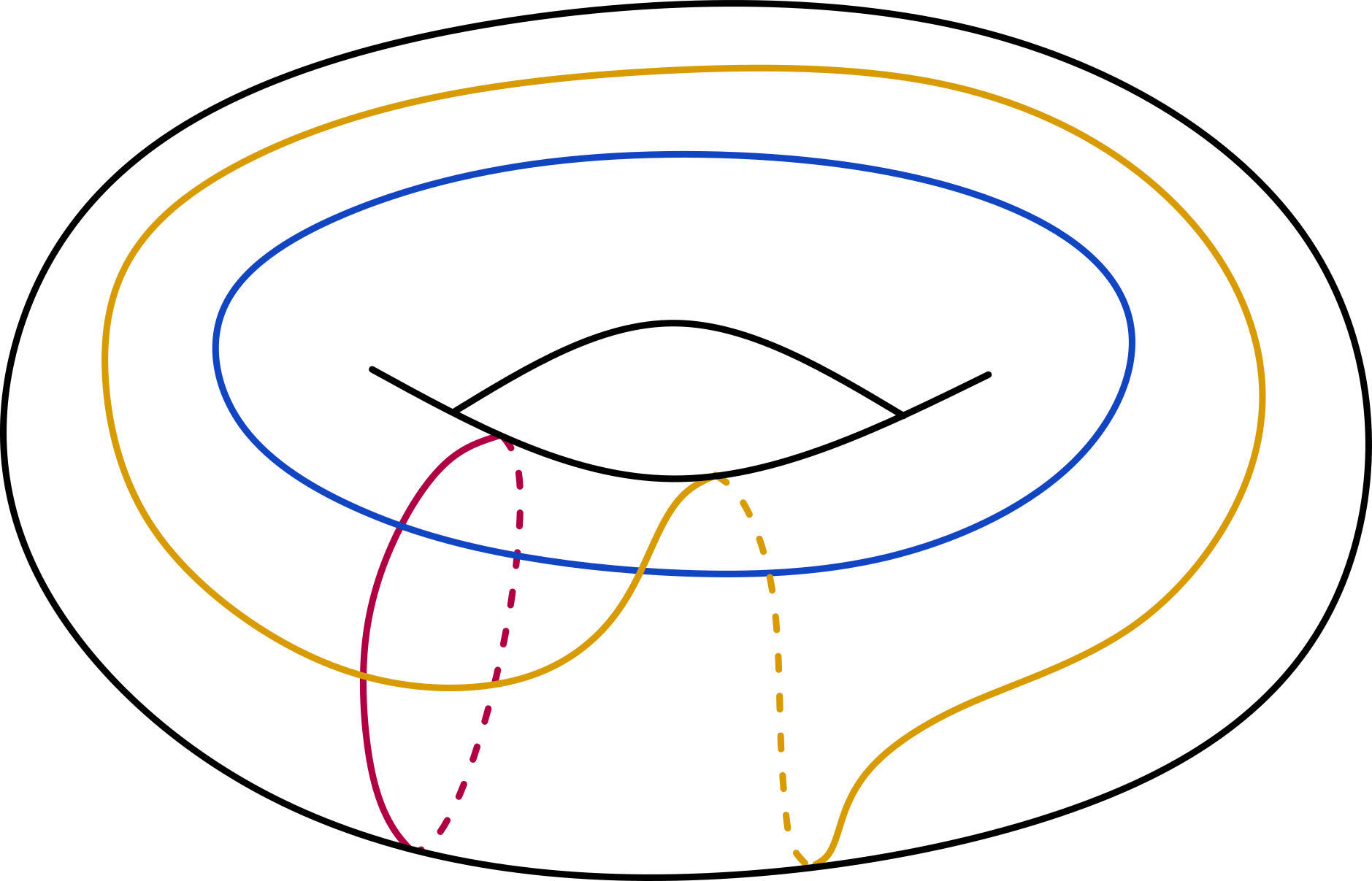}};
    \node at (1.05,0.45) {$a$};
    \node at (0.9,1.8) {$b$};
    \node at (3.2,0.45) {$c$};
    \end{tikzpicture}
\end{center}
\caption{A clique of type necklace}\label{torusfigure}
\end{figure}

\begin{proof}[Proof of Proposition~\ref{prop:torustorusvselse}]
    By Lemma~\ref{properties2and3}, $(u,v)$ is a torus pair if and only if there is a curve $w$ such that $(u,v,w)$ are of type necklace. Lemma~\ref{property1} asserts that the property of being of necklace type is indeed combinatorial, and therefore is preserved by automorphisms of $\onefine(T)$.

    We conclude that torus pairs are preserved by automorphisms of $\onefine(T).$
\end{proof}

\subsection{Pants pairs vs. disjoint pairs}\label{sec:toruspantsvsdisjoint}

It remains to show that we can distinguish pants pairs from disjoint pairs in $\onefine(T).$ We will do this in the following proposition.

\begin{proposition}\label{toruspantsvsdisjoint}
    Let $u$ and $v$ be adjacent curves in $\onefine(T)$ and $\varphi\in \Aut \onefine(T)$. Then $(u,v)$ is a pants pair if and only if $(\varphi(u),\varphi(v))$ is a pants pair.
\end{proposition}

\begin{proof}
    Proposition~\ref{sec:torustorusvselse} ascertains that automorphisms preserve the set of pants and disjoint pairs. Without loss of generality, we may assume that $(u,v)$ is a pants or disjoint pair.

    By Proposition~\ref{prop:torustorusvselse}, we can distinguish whether adjacent curves on a torus are homotopic, which is precisely when they are not a torus pair. Define the \textit{homotopic link} of $(u,v)$, denoted by $\link^{\text{hom}}(u,v)$, to be the subgraph of $\link(u,v)$ induced by curves homotopic to $u$ and $v.$ We will prove the theorem by showing that $(u,v)$ is a pants pair if and only if $\link^{\text{hom}}(u,v)$ is a join.

    \pit{Suppose $(u,v)$ is a pants pair} Then, $u$ and $v$ bound a pinched annulus. Any curve in $\link^{\text{hom}}(u,v)$ is contained in this pinched annulus or in its complement (possibly intersecting $u$ and $v$). Let $w\in \link^{\text{hom}}(u,v)$ be contained in the pinched annulus. Then, $w$ must intersect both $u$ and $v$ at $u\cap v,$ and thus does not intersect them elsewhere. Let $x\in \link^{\text{hom}}(u,v)$ be contained in the complement of the pinched annulus. If $x\cap (u\cup v)$ is nonempty, then $|w\cap x|=1$ if $x\cap u\cap v\neq \emptyset$ and $|w\cap x|=0$ otherwise. We conclude that $x$ and $w$ are adjacent in $\link^{\text{hom}}(u,v),$ which is therefore a join with parts defined by the subsurfaces of $T$ bounded by $u$ and $v$.

    \pit{Suppose $(u,v)$ is a disjoint pair} We will show that $\link^{\text{hom}}(u,v)$ is not a join by showing that no possible partition exists. We have that $u$ and $v$ bound two annuli with boundary, $A$ and $B,$ and every curve in $\link^{\text{hom}}(u,v)$ is contained in precisely one of the two annuli.

    First, we claim that curves contained in the same annulus cannot be in different parts of the partition. Let $a_1,a_2\in \link^{\text{hom}}(u,v)$ be contained in $A.$ Then, there exists a curve $a_3\in\link^{\text{hom}}(u,v)$ contained in $A$ that intersects $a_1$ and $a_2$ at least two times each, and is therefore adjacent to neither.

    Thus, the only possible partition for a join is into two parts, each corresponding to curves contained in one of $A$ or $B.$ However, there is a curve $a$ contained in $A$ that intersects each of $u$ and $v$ once, and a curve $b$ contained in $B$ that intersects $u$ and $v$ at $a\cap u$ and $a\cap v,$ respectively. Therefore, $a$ and $b$ are not adjacent.

    We conclude that there is no possible partition of $\link^{\text{hom}}(u,v)$ into a join.
\end{proof} 

Proposition~\ref{toruspantsvsdisjoint} allows us to combinatorially distinguish between pants pairs and disjoint pairs. With that in hand, we are ready to prove Proposition~\ref{toruscase}.

\subsection{\boldmath{$\Homeo(T) \cong \Aut \onefine(T) $}}

\begin{proof}[Proof of Proposition~\ref{toruscase}]
    We first note that any homeomorphism of $T$ induces an element of $\Aut \onefine(T).$  It remains to show that an element of $\Aut\onefine(T)$ induces a homeomorphism of $S.$
    
    We will reduce this claim to the theorem of Le~Roux--Wolff \cite{transaut} by showing that any automorphism of $\onefine(T)$ induces an automorphism of $\onetrans(T)$. It suffices to show that for any $\varphi \in \Aut\onefine(T)$ and any adjacent curves $u$ and $ v$ in $\onefine(T)$, we have
    \begin{displaymath}
    (u,v) \text{ disjoint } \iff (\varphi(u), \varphi(v)) \text{ disjoint.}
    \end{displaymath}

    In particular, we must show that $(u,v)$ is a pants pair if and only if $(\varphi(u), \varphi(v))$ is a pants pair. By Proposition~\ref{prop:torustorusvselse}, torus pairs are preserved by automorphisms, and by Proposition~\ref{toruspantsvsdisjoint}, pants pairs are distinguishable from disjoint pairs. 

    We therefore have that $\varphi$ induces an element of $\Aut\onetrans(T).$ We now invoke Le~Roux--Wolff \cite[Theorem 1.1]{transaut} that $\Homeo(S_g)\cong \Aut\onetrans(S_g)$ to complete the proof.
\end{proof}

\section{Proof of the main theorem}\label{mainsection}

\begin{proof}[Proof of Theorem \ref{maintheorem}]  There are two cases, depending on the genus $g$. We begin by showing that automorphisms of $\onefine(S_g)$ induce automorphisms of $\onetrans(T)$ (if $g=1)$ or $\fine(S_g)$ (if $g\geq 2$).
    
    \pit{Case 1: $g=1$} This is Proposition~\ref{toruscase}.
    
    \pit{Case 2: $g\geq2$} We observe that homeomorphisms preserve intersection number, so any homeomorphism of $S_g$ induces an automorphism of $\onefine(S_g).$ 
    
    It remains to show that an automorphism of $\onefine(S_g)$ induces a homeomorphism of $S$. We reduce this claim to the result of Long--Margalit--Pham--Verberne--Yao that $\Aut\fine(S_g)\cong \Homeo(S_g)$ via the natural isomorphism. To do this, we show that any automorphism of $\onefine(S_g)$ sends all pairs of vertices corresponding to once-intersecting curves to pairs of curves corresponding to once-intersecting curves.  

    By Proposition~\ref{toruspairaut}, automorphisms preserve torus pairs, and by Proposition~\ref{allpants}, automorphisms preserve pants pairs.  Since an edge can only correspond to a torus pair, a pants pair, or a pair of disjoint curves,  we conclude that any $\varphi \in \Aut\onefine(S_g)$ induces an automorphism of the image of the natural inclusion $\fine(S_g) \hookrightarrow \onefine(S_g)$. We apply the theorem of Long--Margalit--Pham--Verberne--Yao to prove that an automorphism of $\onefine(S_g)$ naturally induces a homeomorphism of $S_g$.

    \medskip\noindent It remains to show that the maps we construct are indeed the inverses of $\Phi.$ For the sake of clarity, we name the maps we use. Let $G$ be $\fine(S_g)$ if $g\geq 2$ and $\onetrans(T)$ if $g=1.$ Let $\Psi_1:\Aut \onefine(S_g)\to \Aut G$ be the map such that $\Psi_1(f)=\overline{f}$ is the automorphism induced by $f.$ We note that $f$ and $\overline{f}$ act the same way on the vertex sets of their corresponding graphs. Let $\Psi_2:\Aut G\to \Homeo (S_g)$ be the map constructed by Le~Roux--Wolff \cite{transaut} (if $g=1$) or Long--Margalit--Pham--Verberne--Yao \cite{fineaut} (if $g\geq 2$). Let $\varphi\in \Homeo(S_g).$ Then,
    \begin{align*}
        \Psi_2\circ \Psi_1\circ \Phi(\varphi) &= \Psi_2\circ \Psi_1(f_{\varphi}), \textrm{ where } f_{\varphi} \textrm{ permutes vertices as prescribed by }\varphi\\
        &= \Phi_2(\overline{f_{\varphi}})\\
        &= \Psi_2(\Psi_2^{-1}(\varphi))\\
        &= \varphi.
    \end{align*}
    Conversely, let $f\in \Aut \onefine(S_g).$ Then,
    \begin{align*}
        \Phi\circ\Psi_2\circ \Psi_1(f) &= \Phi \circ \Psi_2(\overline{f})\\
        &=\Phi(\varphi_{\overline{f}}), \textrm{ where } \varphi_{\overline{f}} \textrm{ permutes curves as prescribed by } \overline{f}\\
        &=\Phi (\varphi_f), \textrm{ since } f \textrm{ and } \overline{f} \textrm{ permute vertices in the same way}\\
        &=f.
    \end{align*}
We conclude that the natural map $\Aut G\to \Homeo(S_g)$ is an isomorphism.
\end{proof}

\bibliographystyle{amsalpha}
\bibliography{mainbib}

\newcommand{\etalchar}[1]{$^{#1}$}
\providecommand{\bysame}{\leavevmode\hbox to3em{\hrulefill}\thinspace}
\providecommand{\MR}{\relax\ifhmode\unskip\space\fi MR }
\providecommand{\MRhref}[2]{%
  \href{http://www.ams.org/mathscinet-getitem?mr=#1}{#2}
}
\providecommand{\href}[2]{#2}
\begin{thebibliography}{LMP{\etalchar{+}}23}

\bibitem[AAC{\etalchar{+}}21]{kcurve}
Shuchi Agrawal, Tarik Aougab, Yassin Chandran, Marissa Loving, J.~Robert
  Oakley, Roberta Shapiro, and Yang Xiao, \emph{{Automorphisms of the k-Curve
  Graph}}, Michigan Mathematical Journal (2021), 1 -- 39.

\bibitem[BHW22]{BHW}
Jonathan Bowden, Sebastian~Wolfgang Hensel, and Richard Webb,
  \emph{Quasi-morphisms on surface diffeomorphism groups}, J. Amer. Math. Soc.
  \textbf{35} (2022), no.~1, 211--231. \MR{4322392}

\bibitem[BM19]{BrendleMargalit}
Tara~E. Brendle and Dan Margalit, \emph{Normal subgroups of mapping class
  groups and the metaconjecture of {I}vanov}, J. Amer. Math. Soc. \textbf{32}
  (2019), no.~4, 1009--1070. \MR{4013739}

\bibitem[Far06]{Farbproblems}
Benson Farb, \emph{Some problems on mapping class groups and moduli space},
  Problems on mapping class groups and related topics, Proc. Sympos. Pure
  Math., vol.~74, Amer. Math. Soc., Providence, RI, 2006, pp.~11--55.
  \MR{2264130}

\bibitem[Iva97]{ivanov}
Nikolai~V. Ivanov, \emph{Automorphism of complexes of curves and of
  {T}eichm\"{u}ller spaces}, Internat. Math. Res. Notices (1997), no.~14,
  651--666. \MR{1460387}

\bibitem[LMP{\etalchar{+}}23]{fineaut}
Adele Long, Dan Margalit, Anna Pham, Yvon Verberne, and Claudia Yao,
  \emph{Automorphisms of the fine curve graph}, Transactions of the AMS
  \textbf{To appear} (2023).

\bibitem[RW22]{transaut}
Frédéric~Le Roux and Maxime Wolff, \emph{Automorphisms of some variants of
  fine graphs}, 2022.

\end{thebibliography}

\end{document}